\documentclass[a4paper,12pt]{amsart}
\usepackage{amsmath}
\usepackage{amsthm}
\usepackage{amscd}
\usepackage{amssymb}
\usepackage{amsfonts}
\usepackage{latexsym}
\usepackage[mathscr]{eucal}
\usepackage{cases}

\newtheorem{thm}{Theorem}[section]
\newtheorem{prop}[thm]{Proposition}
\newtheorem{lemma}[thm]{Lemma}
\newtheorem{cor}[thm]{Corollary}

\def\Im{\mathop{\rm {Im}}\nolimits}

\def\Ker{\mathop{\rm {Ker}}\nolimits}

\begin{document}
\title {Unbounded strongly irreducible operators and transitive representations of quivers
on infinite-dimensional Hilbert spaces
}
\author{Masatoshi Enomoto}
\address[Masatoshi Enomoto]{Institute of Education and Research, Koshien University, Takarazuka, Hyogo 665-0006, Japan}
\email{enomotoma@hotmail.co.jp}
\author{Yasuo Watatani}
\address[Yasuo Watatani]{Department of Mathematical Sciences, 
Kyushu University, Motooka, Fukuoka, 819-0395, Japan}
\email{watatani@math.kyushu-u.ac.jp}

\maketitle

\begin{abstract}
We introduce unbounded strongly irreducible operators and transitive operators. These operators are related to a certain class of indecomposable Hilbert representations of quivers 
on infinite-dimensional Hilbert spaces.
We regard the theory of Hilbert representations of quivers is a generalization of the theory of unbounded operators.
A non-zero Hilbert representation of a quiver is said to be transitive if 
the endomorphism algebra is trivial. 
 If a Hilbert representation of a quiver is transitive, then it is indecomposable. 
But the converse is not true. Let $\Gamma$ be a quiver whose underlying undirected graph is an extended Dynkin diagram.
Then there exists an infinite-dimensional transitive Hilbert representation of $\Gamma$ if and only if $\Gamma$ is not an oriented cyclic quiver.

\medskip


\medskip

\noindent AMS SUBJECT CLASSIFICATION: Primary 47A65, Secondary 

46C07,47A15,16G20.

\end{abstract}

\section{\textbf{Introduction}.}

A bounded linear operator $T$ on a Hilbert space $H$ is called 
strongly irreducible if $T$ cannot be decomposed to a 
non-trivial (not necessarily orthogonal) direct sum of two operators, 
that is, if 
there exist no non-trivial invariant closed subspaces $M$ and $N$ 
of $T$ such that $M \cap N = 0$ and $M + N = H$. 
A strongly irreducible operator is an 
infinite-dimensional generalization of a Jordan block. 
F. Gilfeather \cite{Gi} introduced the notion of strongly irreducible operator .
We refer to 
excellent  books \cite{JW1} and  \cite{JW2} by 
Jiang and Wang on 
strongly irreducible operators. 

 We \cite{EW1},\cite{EW2} studied the relative positions of subspaces in a separable 
infinite-dimensional Hilbert space after  Nazarova \cite{Na1}, Gelfand and 
Ponomarev \cite{GP}. 
We think that relative positions of subspaces have a close relation with 
subfactor theory \cite{Jo},\cite{GHJ}. 
Let $H$ be a Hilbert space and $E_1, \dots E_n$ be $n$ subspaces 
in $H$.  Then it is said that  ${\mathcal S} = (H;E_1, \dots , E_n)$  
is a system of $n$ subspaces in $H$
or a $n$-subspace system in $H$. 
For two systems ${\mathcal S} = (H;E_1, \dots , E_n)$ and ${\mathcal T} = (K;F_1, \dots , F_n)$,  
${\mathcal S}$ and ${\mathcal T}$
are isomorphic if there exists 
an invertible operator $\varphi:H\to K$
such that 
$\varphi(E_{i})=F_{i}$ for
$i=1,2,\cdots,n$.
A non-zero system  
${\mathcal S} = (H;E_1, \dots , E_n)$
is said to be
indecomposable
if it cannot be decomposed to a non-trivial
direct sum
of two
systems up to isomorphism.
We recall that strongly irreducible operators
 contribute an important role
to construct indecomposable systems 
of four subspaces \cite {EW1}.
 
On the other hand Gabriel \cite{Ga} introduced a 
finite-dimensional (linear) representations of quivers by 
attaching vector spaces and linear maps for 
vertices and edges of quivers respectively.   
A finite-dimensional indecomposable representation 
of a quiver is a direct graph generalization of a Jordan block. 
Historically Kronecker\cite{Kro} solved  the indecomposable representations of  $\tilde{A_{1}}$ ,
the so called matrix pencils in 1890.
Nazarova\cite{Na1} and Gelfand-Ponomarev\cite{GP} 
treated the four subspace situation $\tilde{D_{4}}$.
Donovan-Freislich\cite{DF} and Nazarova\cite{Na2}
classified the indecomposable representations of 
the tame quivers.
About these topics we also refer to Bernstein-Gelfand-Ponomarev \cite{BGP}, V. Dlab-Ringel \cite{DR}, Ringel \cite{Ri2},
Gabriel-Roiter \cite{GR}, Kac \cite{Ka},\dots .

We recall infinite-dimensional representations in 
purely algebraic setting. 
In 
\cite{Au} Auslander  found that if a finite-
dimensional algebra 
is not of finite representation type,
then there exist indecomposable modules which are not of finite length.
These are trivially infinite-
dimensional.
Several works about infinite-dimensional Kronecker modules
have been done by N. Aronszjan, A. Dean,U.Fixman ,F.Okoh and F.Zorzitto
 in \cite{Ar},\cite{DZ1}, \cite{F},\cite{FO},\cite{FZ},\cite{Ok}.
A.Dean and F.Zorzitto \cite{DZ2}
constructed a family of infinite-dimensional indecomposable
representations of $\tilde{D_{4}}$.
K.Ringel  \cite{Ri1} founded a general theory of infinite-dimensional  
representations of tame, 
hereditary algebra ( see also \cite{Ri3}, \cite{KR} ).

In \cite{EW3},\cite{E} 
we started to investigate representation theory of quivers 
on Hilbert spaces.
We asked the existence of an indecomposable infinite-dimensional 
Hilbert representation for any quiver whose underlying undirected 
graph is one of extended Dynkin diagrams. 
And we solved it affirmatively 
using the unilateral shift $S$. The argument works even if we replace the unilateral shift $S$ with any strongly irreducible operator.
From this,it is suggested that strong irreducible 
operators  are useful to construct indecomposable Hilbert representations 
of quivers \cite{EW4}. 
 From the analogy of transitive lattice
(see P.R.Halmos\cite{H} and K.J.Harrison,H. Radjavi and P. Rosenthal \cite{HRR}),
we called an indecomposable Hilbert representation $(H,f)$ 
of a quiver such that 
$End(H,f) ={\mathbb C} I$  transitive. If a Hilbert representation of a quiver is transitive, then it is indecomposable. 
But the converse is not true. 
Therefore it is important 
to investigate the existence problem of an transitive infinite-dimensional 
Hilbert representation for any quiver whose underlying undirected 
graph is one of extended Dynkin diagrams.
In this direction, we \cite{EW4} showed  two kinds of constructions of quite non-trivial 
transitive Hilbert representations $(H,f)$ of 
the Kronecker quiver. 

In purely algebraic setting,
a representation of a quiver is called
a brick if its endomorphism ring is a division ring.
But for a Hilbert representation $(H,f)$, 
$End(H,f)$ is a Banach algebra and not 
isomorphic to its purely algebraic endomorphism ring 
in general, because we only consider bounded endomorphisms.
By Gelfand-Mazur theorem, any Banach algebra over $\mathbb{C}$
which is a division ring must 
be isomorphic to
 $\mathbb{C}$.

We remark that  locally  scalar representations of quivers 
 were introduced by 
Kruglyak and Roiter \cite{KRo}. 
But their subject is different from ours.  We also refer to S. Kruglyak, V. Rabanovich and Y. Samoilenko
\cite{KRS} and Y. P. Moskaleva and Y. S. Samoilenko \cite{MS}. 

 We consider finite-dimensional indecomposable representations of 
quivers whose underlying graph is Dynkin diagram.They are transitive (cf.\cite{As}).
 
But it is extremely difficult to solve the existence problem for
infinite-dimensional indecomposable (also transitive) 
Hilbert representations of quivers whose underlying undirected graph is Dynkin diagram.
The existence is not known even for quivers whose underlying undirected graph 
is $D_{4}$.

In this paper 
we introduce unbounded strongly irreducible operators and transitive operators. 
It is known that any unbounded closed operator $T$ on a Hilbert space can be realized as a quotient $BA^{-1}$ of bounded operators $A$ and $B$ on $H$.
This fact is related with operator ranges and intersections of domains of unbounded operators.
See,for example,P.Fillmore and J.Williams
\cite{FiW},W.E.Kaufman\cite{Kau} and
H.Kosaki\cite{Ko}.
We point out that
the study of an unbounded closed operator $T=BA^{-1}$ can be translated to the study of a Hilbert representation given by $A$ and $B$ of the Kronecker quiver. 
We show that  some transitive operators
are constructed by a certain transitive Hilbert representation of the Kronecker quiver. 
We regard the theory of Hilbert representations of quivers is a generalization of the theory of unbounded operators.
We also solve completely 
the existence problem of 
infinite-dimensional transitive
Hilbert representations of quivers
 whose underlying undirected graphs are
the extended Dynkin diagrams.
Let $\Gamma$ be a quiver whose underlying undirected graph 
is an extended Dynkin diagram.
If the underlying undirected graph of $\Gamma$
is not $\widetilde{A_{n}}$,
then there exists an infinite-dimensional transitive Hilbert 
representation of $\Gamma$.
If the underlying undirected graph of $\Gamma$
is  $\widetilde{A_{n}}$,
then there exists an infinite-dimensional transitive Hilbert 
representation of $\Gamma$ if 
and only if $\Gamma$
is not an oriented cyclic quiver.
We used unbounded  transitive operators
based on an idea  
of a transitive lattice by K.J.Harrison,H. Radjavi and P. Rosenthal(\cite{HRR},\cite{RR}).

This work was supported by JSPS KAKENHI 
Grant Number 23654053 and 25287019.

\section{\textbf{Hilbert representations of quivers}}

A quiver $\Gamma =(V,E,s,r)$ is a quadruple consisting of the set $V$ of
vertices, the set $E$ of arrows, and two maps $s,r:E\rightarrow V$ which
associate with each arrow $\alpha \in E$ its support $s(\alpha )$ and range 
$r(\alpha )$. In this paper we assume that  $\Gamma$ is a finite quiver.

 We denote by $\alpha:x\to y$
an arrow with $x=s(\alpha)$
and $y=r(\alpha)$.
Thus a quiver is a directed graph.
We denote by 
$ \left\vert \Gamma \right\vert $ the underlying
undirected graph of a quiver $\Gamma$ .
We say that a quiver $\Gamma $ is
connected if $\left\vert \Gamma \right\vert $ is a connected graph. 
A quiver 
$\Gamma $ is called finite if both $V$ and $E$ are finite sets.
A path of length $m$
is a finite sequence 
$\alpha=(\alpha_{1},\cdots,\alpha_{m})$
of arrows such that $r(\alpha_{k})=s(\alpha_{k+1})$ for
$k=1,\cdots,m-1 $.
Its support is $s(\alpha)=s(\alpha_{1})$
and its range is 
$r(\alpha)=r(\alpha_{m})$.
A path of length $m\geq 1$
is called a cycle if its support and range coincide.
A cycle of length one is called a loop. 
A quiver which is a loop is also called the Jordan quiver $L$.
A quiver which is a cycle of length $m\geq 1$ is also called 
the oriented cyclic quiver $C_{m}$ with length $m\geq 1$.
A quiver is said to be acyclic if it contains no cycles.

\noindent
{\bf Definition.}
Let $\Gamma =(V,E,s,r)$ be a finite quiver. It is said that $(H,f)$ is
a Hilbert representation of  $\Gamma $ if $H=(H_{v})_{v\in V}$ is a family of
Hilbert spaces and $f=(f_{\alpha })_{\alpha \in E}$ is a family of bounded
linear operators $\ f_{\alpha }:H_{s(\alpha )}\rightarrow H_{r(\alpha )}.$

\noindent
{\bf Definition.}
Let $\Gamma =(V,E,s,r)$ be a finite quiver. Let $(H,f)$ and $(K,g)$ be
Hilbert representations of $\Gamma .
$ A homomorphism $T:(H,f)$ $\rightarrow $
$(K,g)$ is a family $T=(T_{v})_{v\in V\text{ }}$of bounded operators $%
T_{v}:H_{v}$ $\rightarrow $ $K_{v}$ satisfying for any arrow $\alpha \in E$
, $T_{r(\alpha )}f_{\alpha }=g_{\alpha }T_{s(\alpha )}.$
 The composition
$T\circ S$ of homomorphisms $T$ and $S$ is defined
by $(T\circ S)_{v}=T_{v}\circ S_{v}$ for $v\in V.$ 
In this way we have obtained a category HRep$%
(\Gamma )$ of Hilbert representations of $\Gamma .$ 
We denote by Hom$%
((H,f),(K,g))$ the set of homomorphisms
$T:(H,f)\to (K,g)$.
We denote by $End(H,f):=Hom((H,f),(H,f))$
the set of endomorphisms.
We can regard 
$End(H,f)$
as a subalgebra of $\oplus_{v\in  V}B(H_{v})$.
In the paper we distinguish the following two classes of operators.
A bounded operator $A$ is said to be a projection(resp. an idempotent) if
$A^{2}=A=A^{*}$(resp.$A^{2}=A$) .
We denote by
\begin{align*}
I&dem(H,f):=\{T\in  End(H,f)\mid T^{2}=T\}\\
=&\{T=(T_{v})_{v\in  V}\in  End(H,f)\mid T_{v}^{2}=T_{v}(\text{for any  }v\in  V)\} \\
\end{align*}
the set of all idempotents of $End(H,f)$.
 Let 0=(0$_{v}$)$_{v\in V}$ be a
family of zero endomorphisms and $I=(I_{v})_{v\in V}$ be a family of
identity endomorphisms. 
It is said that $(H,f)$ and $(K,g)$ are isomorphic,
denoted by $(H,f)\cong $ $(K,g),$ if there exists an isomorphism $\varphi :$
$(H,f)\rightarrow (K,g)$, that is,there exists a family $\varphi =(\varphi
_{v})_{v\in V}$ of bounded invertible operators $\varphi _{v}\in
B(H_{v},K_{v})$ such that ,for any arrow $\alpha \in E,\varphi _{r(\alpha
)}f_{\alpha }=g_{\alpha }\varphi _{s(\alpha )}.$ We say that $(H,f)$ is a
finite-dimensional representation if $H_{v}$ is finite-dimensional for all $%
v\in V.$ And $(H,f)$ is an infinite-dimensional representation if $H_{v}$ is
infinite-dimensional for some $v\in V.$

We recall a notion of indecomposable representation 
in \cite{EW3}
that is, a representation which cannot be 
decomposed into a direct sum of smaller representations anymore.

\noindent
{\bf Definition.}
Let $\Gamma =(V,E,s,r)$ be a finite quiver. Let $(K,g)$ and $(K^{^{\prime
}},g^{^{\prime }})$ be Hilbert representations of $\Gamma .$ 
We define the
direct sum $(H,f)$ $=(K,g)$ $\oplus $ $(K^{^{\prime }},g^{^{\prime }})$ by $%
H_{v}=K_{v}\oplus $ $K_{v}^{^{\prime }}$(for $v\in V)$ and $f_{\alpha
}=g_{\alpha }\oplus $ $g_{\alpha }^{^{\prime }}$(for $\alpha \in E).$ 
It is said
that a Hilbert representation $(H,f)$ is zero, denoted by $(H,f)$ $=0$ if 
$H_{v}=0$ 
$\text{for any }$ $v\in V.$

\noindent
{\bf Definition.}
A Hilbert representation $(H,f)$ of $\Gamma$ is said to be decomposable
if $(H,f)$ is isomorphic to a direct sum of two non-zero Hilbert representations.
A non-zero Hilbert representation $(H,f)$ of $\Gamma$ is called indecomposable if it is not decomposable,
that is,
if $(H,f)\cong (K,g)\oplus (K^{\prime},g^{\prime})$
then $(K,g)\cong 0$ or $(K^{\prime},g^{\prime})\cong 0$.

The following proposition is useful to show the indecomposability 
in concrete examples.
\begin{prop}\cite[Proposition 3.1.]{EW3}
Let $(H,f)$ be a Hilbert representation 
of a quiver $\Gamma$.
Then the following conditions 
are equivalent:
\begin{enumerate}
\item $(H,f)$ is indecomposable.
\item $ Idem(H,f) = \{0,I\}$. 
\end{enumerate}

\end{prop}

\noindent
{\bf Remark.}
The indecomposability of Hilbert representations of a quiver 
is an isomorphic invariant, but it is not a unitarily equivalent invariant.
Hence we cannot replace the set $Idem(H,f)$ of idempotents of endomorphisms by the subset of idempotents of endomorphisms which  consists of projections to show the indecomposability.

\noindent
{\bf Definition.}(\cite[page 569]{EW4})
A Hilbert representation $(H,f)$ of a quiver $\Gamma$ is said to be 
 {\it transitive} 
if $End(H,f) = {\mathbb C}I$. 
If a Hilbert representation $(H,f)$ of $\Gamma$ is 
transitive , then $(H,f)$ is indecomposable. In fact, since 
$End(H,f) = {\mathbb C}I$, any 
idempotent endomorphism $T$ is $0$ or $I$. In purely algebraic 
setting, a representation of a quiver is said to be a {\it brick} if 
its endomorphism ring is a division ring(see for example, cf.\cite{As}). 

\bigskip
Let $H$ be a Hilbert space and $E_1, \dots E_n$ be $n$ subspaces 
in $H$.  Then it is said that  ${\mathcal S} = (H;E_1, \dots , E_n)$  
is a system of $n$ subspaces in $H$.  
Let ${\mathcal T} = (K;F_1, \dots , F_n)$  
be  another system of $n$ subspaces in a Hilbert space $K$. 
Then we say that 
$\varphi : {\mathcal S} \rightarrow {\mathcal T}$ is a 
homomorphism if $\varphi : H \rightarrow K$ is a bounded linear 
operator satisfying that  
$\varphi(E_i) \subset F_i$ for $i = 1,\dots ,n$. We say that
$\varphi : {\mathcal S} \rightarrow {\mathcal T}$
is an isomorphism if $\varphi : H \rightarrow K$ is 
an invertible (i.e., bounded  bijective) linear 
operator satisfying that  
$\varphi(E_i) = F_i$ for $i = 1,\dots ,n$. 
It is said that systems ${\mathcal S}$ and ${\mathcal T}$ are 
{\it isomorphic} if there is an isomorphism  
$\varphi : {\mathcal S} \rightarrow {\mathcal T}$. This means 
that the relative positions of $n$ subspaces $(E_1, \dots , E_n)$ in $H$ 
and   $(F_1, \dots , F_n)$ in $K$ are same under disregarding angles. 
Let us denote by 
$Hom(\mathcal S, \mathcal T)$ the set of homomorphisms of 
$\mathcal S$ to $\mathcal T$ and  
$End(\mathcal S) := Hom(\mathcal S, \mathcal S)$ 
the set of endomorphisms on $\mathcal S$. 
Let ${\mathcal S} = (H;E_1, \dots , E_n)$ and 
 $\mathcal S^{\prime}=
 (H^{\prime};E_{1}^{\prime},\cdots,E_{n}^{\prime})$ be 
systems of $n$ subspaces in Hilbert spaces 
 $H$ and $H^{\prime}$.  Then their direct sum 
$\mathcal {S} \oplus\mathcal {S}^{\prime}$ is defined by 
\[
\mathcal {S} \oplus\mathcal {S}^{\prime}
:= (H\oplus H^{\prime};
E_{1}\oplus E_{1}^{\prime},\dots,E_{n}\oplus E_{n}^{\prime}).
\]

 A system $\mathcal S=(H;E_{1},\dots,E_{n})$
of $n$ subspaces is said to be {\it decomposable} 
if the system $\mathcal S$ is isomorphic to 
a direct sum of two non-zero systems.  
A non-zero system $\mathcal S=(H;E_{1},\cdots,E_{n})$ 
of $n$ subspaces is called 
{\it indecomposable} if it is not decomposable.

We recall that strongly irreducible operators $A$ play an extremely
important role to construct indecomposable systems of 
four subspaces. Moreover the commutant $\{A\}'$ corresponds to 
the endomorphism ring. 

For any single operator $A \in B(K)$ on a Hilbert space $K$, let
$\mathcal{S}_A = (H;E_1,E_2,E_3,E_4)$
be the associated  operator system such that 
$H = K \oplus K$ and 
\[
E_{1}=K\oplus 0,
E_{2}=0\oplus K, 
E_{3}=\{(x,Ax); x\in K\},
E_{4}=\{(y,y); y\in K\}. 
\]
It follows that
$$
End (\mathcal{S}_{A}) = \{ T \oplus T \in B(H) ; 
T \in B(K), \ AT = TA \}
$$
is isomorphic to the commutant $\{A\}^{\prime}$. 
The associated system $\mathcal{S}_A$ of four subspaces 
is indecomposable if and only if $A$ is strongly irreducible. 
Moreover for any operators $A,B  \in B(K)$ on a Hilbert space $K$, 
the associated systems $\mathcal{S}_A$ and $\mathcal{S}_B$ 
are isomorphic if and only if $A$ and $B$ are similar. 

Following after \cite{H} and \cite{HRR}, we \cite[page 272]{EW1} introduced a transitive system of subspaces.
A system ${\mathcal S}=(H;E_{1},E_{2},$ $\cdots,E_{n})$ of  $n$ subspaces
in a Hilbert space is called transitive 
if the endomorphism algebra is trivial, that is,
$$
End({\mathcal S})=\{A \in B(H) ; \ A(E_{i}) \subset E_{i} 
\text{ for any } i=1,2,\cdots,n\} = {\mathbb C}I.
$$

\section{\textbf{Unbounded strongly irreducible operators}.}
In this section we shall introduce 
unbounded strongly irreducible operators and transitive operators.
These operators are related to a certain class of indecomposable Hilbert representations of quivers on infinite-dimensional Hilbert spaces and four- subspace systems.
Let $H$ be a Hilbert space and $A$ a 
bounded linear operator on $H$.
We denote the image of $A$ by $\Im (A)$
and the graph of $A$ by $G(A)$, that is,
$G(A)=\{(x,Ax);x\in H\}$.
For elements $x,y\in H$,
we denote a rank one operator $\theta_{x,y}$
by $\theta_{x,y}(z)=(z\vert y)x$ for $z\in H$.
\noindent  

 P.R.Halmos \cite{H} initiated the study of transitive lattices.
A lattice ${\mathcal L}$ of subspaces 
of a Hilbert space $H$ containing $0$ and $H$ is  called a 
transitive lattice if 
$$
\{A \in B(H) ; \ AM \subset M 
\text{ for any } M \in {\mathcal L}\} = {\mathbb C}I.
$$
K.J.Harrison,H. Radjavi and P. Rosenthal (\cite{HRR}) 
constructed a transitive subspace lattice 
using an unbounded weighted shift as follows:
Let $K=\ell^2(\mathbb{Z})$ be a Hilbert space with an orthogonal basis
$\{e_{i}\}_{i=-\infty}^{+\infty}.$
Let $$w_{n}=1 \ \  (n\leq 0),\ \ w_{n}=exp((-1)^{n}n!) \ \ (n>0).$$
Let $T$ be the bilateral weighted shift 
defined by $Te_{n}=w_{n}e_{n+1}$,
with the domain  
$$D(T)=\{x=\sum_{i=-\infty}^{+\infty}
\alpha_{i}e_{i}
;
\sum_{i=-\infty}^{+\infty}
\vert\alpha_{i}w_{i}\vert^2<+\infty\}.$$
Put $E_{1}=K\oplus 0,
E_{2}=0\oplus K,
E_{3}=G(T),
E_{4}=\{(x,x);x\in K\}$.
Their transitive lattice 
is
$\mathcal{L}=\{0,H=K\oplus K
,E_{1},E_{2},E_{3},E_{4}\}$. 
See also a book Radjavi-Rosenthal \cite[4.7. page 78]{RR}. 

We \cite{EW4} considered a finite subspace lattice
as a Hilbert representation of a quiver $\Gamma$ as follows:
Let ${\mathcal L} = \{0,M_1,$ $M_2, \dots, M_n, H\}$
be a finite lattice.  Consider a $n$ subspace quiver $R_n=(V,E,s,r)$, 
that is, $V = \{1,2,\dots,n,n+1 \}$ and 
$E = \{\alpha_k ; \ k = 1, \dots, n \}$ with $s(\alpha_k) = k$ 
and $r(\alpha_k) = n+1$ for $k = 1, \dots, n$. 
Then there exists a Hilbert representation $(K,f)$ of $R_n$ such that 
$K_k = M_k$, $K_{n+1} = H$ and $f_{\alpha_k} : M_k \rightarrow H$ is an 
inclusion for  $k = 1, \dots, n$. The lattice ${\mathcal L}$ is 
transitive  if and only if the corresponding Hilbert representation 
$(K,f)$ is transitive. 
By this fact we may use the terminology "transitive" 
in the Hilbert representation case.


We recall some facts on strongly irreducible operators for
convenience.

\begin{lemma}
Let $A$ be a bounded operator on a Hilbert space $H$.
Then the following three conditions are equivalent:

\begin{itemize}
\item [(0)]
For any closed subspaces  $M$ and $N$ of $H$ with $H=M+N$ and $M\cap N=0$,
if $AM\subset M$ and $AN\subset N$, then $M=0$ or $N=0$.

\item [(1)] If $T\in B(H)$ is an idempotent in the commutant $\{A\}^{\prime}$ of $A$, then $T=0$ or $T=I$,
\item [(2)] 
If $T\in B(H) $ is an idempotent such that $(T\oplus T)(G(A))\subset G(A)$ , then $T=0$ or $T=I$.
 
\end{itemize}

\end{lemma}
\begin{proof} 
Let  $M$ and $N$ be closed subspaces of $H$ such that $H=M+N$ and $M\cap N=0$, then there exists an idempotent $E$ such that $M=E(H)$ and
$N=(I-E)H$. Hence $(0)$ is equivalent to (1).
We shall show that $(1)$ is equivalent to (2).
Assume that (1) holds.
Let $T\in B(H)$ be an idempotent such that $(T\oplus T)(G(A))\subset G(A)$.
Then for any $x\in H$,
there exists $y\in H$ such that 
$(T\oplus T)((x,Ax))=(y,Ay)$.
Hence 
$Tx=y$ and $TAx=Ay$.
Thus $TA=AT$.
Hence $T\in \{A\}^{\prime}$. 
Since  $T$ is an idempotent 
, $T=0$ or $T=I$.
Hence (2) holds.
Next we assume that (2) holds.
Take an idempotent $T\in \{A\}^{\prime}\cap B(H)$.
Then $$(T\oplus T)((x,Ax))=(Tx,TAx)=(Tx,ATx).$$
Thus 
$(T\oplus T)(G(A))\subset G(A)$.
We have $T=0$ or $T=I$.
Hence (1) holds.
\end{proof}

\noindent
{\bf Definition.}
A bounded operator $A\in B(H)$ is said to be strongly irreducible if $A$ satisfies one of the three conditions of the above lemma.

Inspired by the example of K.J.Harrison,H. Radjavi and P. Rosenthal
we introduce unbounded strongly irreducible operators and
unbounded transitive operators.

\noindent 
{\bf Definition.}
Let $A$ be an unbounded closed operator on a Hilbert space $H$ with the domain $D(A)\subset H$.
We define  the (bounded) commutant $\{A\}^{\prime}$ of $A$  by

$\{A\}^{\prime}
=\{S\in B(H)$
$; S(D(A))\subset D(A) \text{ and,
for any } x\in D(A), ASx=SAx\}$.
See for example \cite[\S 17]{Ak}.
Let $A$ and $B$ be unbounded closed operators on $H$.
We say that $A$ and $B$ are similar if there exists a bounded invertible operator $T\in B(H)$ such that
$T(D(A))=D(B)$ and $B=TAT^{-1}$.
We say that $A$ is an orthogonal direct sum $A_{1}\oplus A_{2}$
of operators $A_{1}$ and $A_{2}$ on $H=H_{1}\oplus H_{2}$
if $D(A)=\{(x_{1},x_{2});x_{1}\in D(A_{1}),x_{2}\in D(A_{2})\}$
and $Ax=(A_{1}x_{1},A_{2}x_{2})$ for $x=(x_{1},x_{2})\in D(A).$

\begin{lemma}
Let $A$ be an unbounded closed operator on a Hilbert space $H$ with the domain $D(A)\subset H$.
Then the following three conditions are equivalent:

\begin{itemize}
\item[(0)]If $A$ is similar to $A_{1}\oplus A_{2}$
on $H=H_{1}\oplus H_{2}$ for some unbounded closed operators 
$A_{1}$ and $A_{2}$, then $H_{1}=0$ or $H_{2}=0$.
\item[(1)] For any idempotent $E\in B(H)$, if $E$ is in
 the commutant $\{A\}^{\prime}$,
then $E=0$ or $E=I$.

\item[(2)] For any idempotent $E\in B(H)$, if
$(E\oplus E)(G(A))\subset G(A)$,
 then $E=0$ or $E=I$.
\end{itemize}
\label{lemma:idem}
\end{lemma}
\begin{proof}
We shall show that (0)$\Rightarrow$(1).
Let $E\in \{A\}^{\prime}$ be an idempotent.
We have $E(D(A))\subset D(A)$ and
$AEx=EAx$ for $x\in D(A)$.
There exists an invertible operator
$T\in B(H)$ such that
$T(E(H))=H_{1}$ and $T((I-E)H)=H_{2}$
and
$H=H_{1}\oplus H_{2}$.
We define
${A_{1}}x=TAT^{-1}x=TAET^{-1}x$ for $x\in T(E(D(A)))\subset H_{1}$.
Since $E(D(A))\subset D(A)$,$A_{1}$ is well defined.
And ${A_{1}}$ is an operator from $T(E(D(A)))$
 to $H_{1}$ by $AEx=EAx$ for $x\in D(A)$.
We define ${A_{2}}x=TAT^{-1}x=TA(I-E)T^{-1}x$ for $x\in T((I-E)(D(A)))\subset H_{2}$.
Since $E(D(A))\subset D(A)$,$A_{2}$ is well defined.
And ${A_{2}}$ is an operator from $T((I-E)(D(A)))$
 to $H_{2}$ by $AEx=EAx$ for $x\in D(A)$.
Hence we have 
$$TAT^{-1}=
TAET^{-1}+TA(I-E)T^{-1}
={A_{1}}\oplus {A_{2}}.$$
Hence $A\cong {A_{1}}\oplus {A_{2}}$ on $H_{1}\oplus H_{2}$.
Since (0) holds, we have $H_{1}=0$ or $H_{2}=0$.
Hence $TE(H)=0$ or $T(I-E)(H)=0$.
So $E=0$ or $E=I$.
Thus we have 
(0)
$\Rightarrow$(1).
Conversely we shall show that (1)$\Rightarrow$(0).
Assume that
$A\cong {A_{1}}\oplus {A_{2}}$ on $H_{1}\oplus H_{2}$ for some unbounded closed operators 
$A_{1}$ and $A_{2}$.
There exists an invertible operator $T\in B(H)$ such that
$TAT^{-1}x= ({A_{1}}\oplus {A_{2}})x$ for $x\in D( {A_{1}}\oplus {A_{2}})=T(D(A))$.
There exists an idempotent $E\in B(H)$
such that $T^{-1}H_{1}=E(H)$ and $T^{-1}H_{2}=(I-E)H$.
We shall show that
$E(D(A))\subset D(A)$ and $AE=EA$ on $D(A)$.
We have
$T^{-1}D(A_{1})\subset T^{-1}H_{1}=EH$
and
$T^{-1}D(A_{2})\subset T^{-1}H_{2}=(I-E)H.$
$D(A)
=T^{-1}D({A_{1}}\oplus {A_{2}})
=T^{-1}D({A_{1}})+T^{-1}D({A_{2}})$.
$$E(D(A))
=E(T^{-1}D({A_{1}})+T^{-1}D({A_{2}}))
=T^{-1}D({A_{1}})$$
$$\subset T^{-1}D({A_{1}})+T^{-1}D({A_{2}})=D(A).$$
For $x\in D(A),
x=x_{1}+x_{2},x_{1}\in T^{-1}D({A_{1}}),
x_{2}\in T^{-1}D({A_{2}})$,
we have 
$AEx=(T^{-1} ({A_{1}}\oplus {A_{2}})T)E(x_{1}+x_{2})
=(T^{-1} ({A_{1}}\oplus {A_{2}})T)x_{1}
=T^{-1}{A_{1}}Tx_{1}.$
And $EAx=E(T^{-1} ({A_{1}}\oplus {A_{2}})T)(x_{1}+x_{2})
=E(T^{-1}{A_{1}}Tx_{1}+ T^{-1}{A_{2}}Tx_{2})
=T^{-1}{A_{1}}Tx_{1}.$
Thus we have $AE=EA$ on $D(A)$.
Therefore $E=0$ or $E=I$.
Hence $H_{1}=0$ or $H_{2}=0$.
Next, we shall show that (1)
$\Rightarrow$(2).
Let $E\in B(H)$ be an idempotent such that 
$(E\oplus E)(G(A))\subset G(A)$.
Then for any $x\in D(A)$,
there exists $y\in D(A)$ such that
$(E\oplus E)(x,Ax)=(y,Ay)$.
Hence 
$$(Ex,EAx)=(y,Ay)=(Ex,AEx).$$
Thus $E\in \{A\}^{\prime}$.
By (1),
then $E=0$ or $E=I$.
Conversely,we shall show that (2)
$\Rightarrow$(1).
Let $E\in \{A\}^{\prime}$ be an idempotent.
Hence $E(D(A))\subset D(A)$
and
$EAx=AEx$ for $x\in D(A)$.
$$(E\oplus E)((x,Ax))=(Ex,EAx)=(Ex,AEx).$$
Hence 
$(E\oplus E)(G(A))\subset G(A)$.
then $E=0$ or $E=I$.
\end{proof}
\noindent
{\bf Definition.}
An  unbounded closed operator $A$ is said to be strongly irreducible if $A$ satisfies  one of the three conditions of the above lemma.
 The next lemma is proved similarly.

\begin{lemma}
Let $A$ be an unbounded closed operator on a Hilbert space $H$ with the domain $D(A)\subset H$.
Then the following two conditions are equivalent:
\begin{itemize}
\item[(1)] For any $T\in B(H)$, if $T$ is in
 the commutant $\{A\}^{\prime}$,
then $T$ is a scalar operator.

\item[(2)] For any $T\in B(H)$, if
$(T\oplus T)(G(A))\subset G(A)$, then $T$ is a scalar operator.
\end{itemize}
\label{transitive operator}
\end{lemma}

\noindent
{\bf Definition.}
An unbounded closed operator $A$ is said to be 
transitive if
$A$ satisfies one of the two conditions 
of the above lemma.

If an unbounded closed operator $A$ is transitive,
then $A$ is strongly irreducible.
Any bounded strongly irreducible operator $A$ on a Hilbert space 
$H$ with $\dim H\geq 2$ is not transitive,
because $A\in \{A\}^{\prime}$.

By the same argument we have the following lemma.

\begin{lemma}
Let $A$ be an unbounded closed operator on a Hilbert space $K$ with the domain $D(A)$.
Let 
 ${\mathcal S}_{A}
=(H;E_{1},E_{2},E_{3},E_{4})$
be a four-subspace system
such that
$H=K\oplus K,
E_{1}=K\oplus 0,
E_{2}=0\oplus K,
E_{3}=\{(x,Ax);x\in D(A)\},
E_{4}=\{(x,x);x\in K\}.$
Then 
${\mathcal S}_{A}$ is transitive 
if and only if $A$ is transitive.
\label{lemma:transitive}
\end{lemma}

We shall construct transitive operators using transitive Hilbert representations and quotients of operators.

\noindent
{\bf Definition.}
Let $A$ and $B$ be bounded linear operators on a Hilbert space ${H}$.
We say that $B(A\vert_{\Ker(A)^{\perp}})^{-1}$ is a quotient 
of $B$ by $A$. We denote $(A\vert_{\Ker(A)^{\perp}})^{-1}$ 
briefly by $A^{-1}$.
If we have an additional condition such that  $\ker A\subset \ker B$,
then the quotient is the mapping $Ax\mapsto Bx,x\in {H}.$
In \cite{Kau}, Kaufman showed the
following useful result about quotient operators. 
\begin{thm}\cite[Theorem 1,page 531]{Kau}
Let $T$ be an unbounded operator on a Hilbert space $H$.
Then  $T$ is a closed operator if and only if 
$T=B(A\vert_{\Ker(A)^{\perp}})^{-1}$ for some $A,B\in B(H)$ such 
that $\Im (A^{*})+\Im(B^{*})$ is closed in $H$.

\label{thm,Kaufman}
\end{thm}

We show that  there is a non-zero
surjective algebra homomorphism of the endomorphism algebra of a Hilbert representation of the Kronecker quiver to the endomorphism algebra of a four-subspace system. 
The Kronecker quiver $Q$ is a quiver with two vertices $\{1,2\}$ and 
two paralleled arrows $\{\alpha, \beta\}$:
$$
Q : 1 ^{\overset {\alpha}{\longrightarrow}}
_{\underset {\beta}{\longrightarrow}} 2
$$
A Hilbert representation $(H,f)$ of 
the Kronecker quiver is given 
by two Hilbert spaces $H_1$, $H_2$ and two bounded operators 
$f_{\alpha}, f_{\beta}: H_1 \rightarrow H_2$. 

\begin{prop}
Let $K\ne0$ be a Hilbert space and $A,B\in B(K)$.
Let $(H,f)$ be a Hilbert representation of the Kronecker quiver $Q$ such that
$H_{1}=H_{2}=K,$
$f_{\alpha}=A$ and $f_{\beta}=B$.
Let $\mathcal S =(E_{0};E_{1},E_{2},
E_{3},E_{4})$ be a four-subspace system 
such that $E_{0}=K\oplus K,E_{1}=K\oplus 0,
E_{2}=0\oplus K,
E_{3}=\{(Ax,Bx);x\in K\},
E_{4}=\{(x,x);x\in K\}$.
Assume that $E_{3}$ is closed.
Then there exists a non-zero
surjective algebra homomorphism $\Phi$
of $End(H,f)$ to $End(\mathcal S)$.
Moreover, if $\ker A \cap \ker B = 0$, then $\Phi$ is one to one. 
\label{Prop,surjective}
\end{prop}
\begin{proof}
Let $(S,T)$ be in $End(H,f)$.
We have $AS=TA$ and $BS=TB$.
Since $(T\oplus T)(Ax,Bx)
=(TAx,TBx)=(ASx,BSx),$
hence $(T\oplus T)(E_{3})\subset E_{3}$.
Clearly $(T\oplus T)(E_{i})\subset E_{i}$
for $i=1,2,4$.
Thus we have that $T\oplus T$ is in $End(\mathcal S)$.
We define a mapping $\Phi$
of $End(H,f)$ to $End(\mathcal S)$
by $\Phi(S,T)=T\oplus T.$
The map $\Phi$
is an algebra homomorphism.
We shall show that the map $\Phi$
is onto.

Take $C
\in End(\mathcal S)$.
Then
there exists $T\in B(K)$ such that $C=(T\oplus T)$.
We have that
$$(T\oplus T)\{(Ax,Bx);x\in K\}
\subset \{(Ay,By);y\in K\}.$$
Hence, for any $x\in K$,
there exists $y\in K$
such that
$TAx=Ay$ and $TBx=By$.
We put $L_{0}=\ker A
\cap \ker B$ and
$L_{1}=L_{0}^{\perp}\cap K$.
By a decomposition of 
$y$ such that  $y=y_{0}+y_{1},y_{0}\in L_{0},y_{1}\in L_{1}$,
we have
$TAx=Ay_{1},TBx=By_{1}$.
We define an operator 
$S$ by
$Sx=y_{1}$.
We shall show that $S$ is well defined.
If there exists another $y^{\prime}=
y^{\prime}_{0}+y^{\prime}_{1}\in K$
for $y^{\prime}_{0}\in L_{0}$ and 
$y^{\prime}_{1}\in L_{1}$ 
such that 
$TAx=Ay^{\prime}=Ay^{\prime}_{1}$ and $TBx=By^{\prime}=By^{\prime}_{1}$.
We have
$Ay_{1}=Ay^{\prime}_{1}$
and
$By_{1}=By^{\prime}_{1}$.
Hence
$y_{1}-y^{\prime}_{1}
\in (\ker A\cap\ker B)=L_{0}$.
We also have
$y_{1}-y^{\prime}_{1}
\in L_{1}$.
Hence 
$y_{1}-y^{\prime}_{1}
\in L_{0}\cap L_{1}=(0)$.
So $y_{1}=y^{\prime}_{1}$.
Thus $S$ is well defined.
Clearly $S$ is  linear.
We shall show that $S$ is a closed operator.
Assume that 
$x_{n}\to x$
and $Sx_{n}=y_{n,1}\to y_{1},$ for $x_{n},x\in K$ and $y_{n,1},y_{1}\in L_{1}.$
Since $Sx_{n}=y_{n,1}$,
we have
that $TAx_{n}=Ay_{n,1}\to Ay_{1}$ and $TBx_{n}=By_{n,1}\to By_{1}.$
If $n\to \infty$,
then $$TAx=Ay_{1}\text{ and }TBx=By_{1}.$$
It follows that $Sx=y_{1}$.
Therefore $S$ is closed.
Hence $S$ is bounded.

Since $TAx=Ay_{1}=ASx$ and $TBx=By_{1}=BSx$
for $x\in K$ and $y_{1}\in L_{1}$,
 we have that 
$$TA=AS \text{ and }TB=BS.$$
Hence $(S,T)\in End(H,g)$.
And $\Phi(S,T)=T\oplus T$.
Hence $\Phi$
is surjective.
We shall show that if $\ker A \cap \ker B = 0$, then $\Phi$ is one to one.
Suppose that  $\Phi(S,T)=T\oplus T=0$ for $(S,T)\in End(H,f)$.
Then $T=0$.
We have that
for any $x\in K$,
$$ASx=TAx=0,BSx=TBx=0.$$
Hence $Sx\in \ker A\cap \ker B=0$.
Since $Sx=0$ for any $x\in K$,
we have $S=0$.
Thus $(S,T)=0$. Therefore $\Phi$ is one to one.

\end{proof}
\noindent
{\bf {Remark.}}
Let $K$ be a Hilbert space and $A,B\in B(K)$.
We consider $$Z=
\begin{pmatrix}
  A \\
  B \\
\end{pmatrix}
:K\to K\oplus K \text{ and }
Zx=(Ax,Bx) \text{ for } x\in K.$$
We have 
$$Z^*=(A^*,B^*):K\oplus K\to K
\text{ and }
Z^*\begin{pmatrix}
  x \\
  y \\
\end{pmatrix}
=A^*x+B^*y \text{ for } x,y\in K.$$
Since $\Im(Z)$ is closed if and only if $\Im(Z^*)$ is closed,
we have that
$$\{(Ax,Bx);x\in K\} \text { is closed
if and only if } \Im(A^*)+\Im(B^*)\text{ is closed}.$$

\noindent
{\bf {Remark.}}
The map $\Phi$ is not one-one in general.
We shall give an example $\Phi$ which is not one to one.
Let $K$ be a Hilbert space and
$A,B$ be operators on $K\oplus K$ such that
$A=B=
\begin{bmatrix} 
1 & 0 \\ 
0 & 0 \\ 
\end{bmatrix} .$
Let $S_{1},T_{1},S_{2},T_{2}$ be operators on $K\oplus K$ such that  
$S_{1}
=
\begin{bmatrix} 
0 & 0 \\ 
1 & 0 \\ 
\end{bmatrix},
T_{1}
=
\begin{bmatrix} 
0 & 0 \\ 
0 & 1 \\ 
\end{bmatrix}$
and 
$S_{2}
=
\begin{bmatrix} 
0 & 0 \\ 
0 & 1 \\ 
\end{bmatrix},
T_{2}
=
\begin{bmatrix} 
0 & 0 \\ 
0 & 1 \\ 
\end{bmatrix}$.
Then
$(S_{1},T_{1})$ and $(S_{2},T_{2})$ are in $End(H,f)$.
And $(S_{1},T_{1})$ and $(S_{2},T_{2})$
give the same endomorphism 
$T_{1}\oplus T_{1}$ of $\mathcal S$.
Thus $\Phi$ is not one to one.

Under a certain condition we have a correspondence
between transitive Hilbert representations of the Kronecker quiver 
and transitive operators. 
\begin{prop}
Let $K$ 
be a Hilbert space and $A,B\in B(K)$.
Assume that
$\ker A=0
\text{ and } \Im A^{*}+\Im B^{*}$  is closed in $K.$
Let $(H,f)$ be a Hilbert representation of the Kronecker quiver $Q$ such that
$H_{1}=H_{2}=K,$
$f_{\alpha}=A$ and $f_{\beta}=B$.
Then 
  $BA^{-1}$ is transitive  if and only if
 $(H,f)$ is transitive.
\label{unbounded:transitive} 
\end{prop}
\begin{proof}
At first we note that the graph 
$G(BA^{-1})=\{(Ax,Bx);x\in K\}$,   
because  $\ker (A)=0$. 
Since 
$\Im A^{*}+\Im B^{*}$  is closed,  
the operator $BA^{-1}$ is
a closed operator 
by Remark after  Proposition \ref{Prop,surjective}
(or Theorem \ref{thm,Kaufman}).
Let ${\mathcal S}_{BA^{-1}}=(E_{0};E_{1},E_{2},
E_{3},E_{4})$ be a four-subspace system 
such that $E_{0}=K\oplus K,E_{1}=K\oplus 0,
E_{2}=0\oplus K,
E_{3}=\{(Ax,Bx);x\in K\} = G(BA^{-1}),
E_{4}=\{(x,x);x\in K\}$. Since $\ker (A)=0$, 
there exists an algebra isomomorphism $\Phi$
of $End(H,f)$ onto $End({\mathcal S}_{BA^{-1}})$ by 
Proposition
\ref{Prop,surjective}. Therefore $(H,f)$ is transitive 
if and only if ${\mathcal S}_{BA^{-1}}$ is transitive. 
Moreover  ${\mathcal S}_{BA^{-1}}$ is transitive 
if and only if  $BA^{-1}$ is transitive by 
Lemma \ref{lemma:transitive}. This implies the conclusion.  
\end{proof}

In the following we shall give some examples of transitive operators.

\begin{prop}
Let $Q$ be the Kronecker  quiver.
Let $S$ be the  unilateral shift on 
${H} =\ell^{2}(\mathbb{N})$ with a canonical 
basis $\{e_{1},e_{2},...\}$. 
For a bounded weight vector 
${\lambda}=(\lambda_{1},\lambda_{2},...)\in \ell^{\infty}(\mathbb{N})$ 
we associate with a diagonal operator 
$D_{\lambda}=\text{diag}(\lambda_{1},\lambda_{2},...)$, 
so that $SD_{\lambda}$ is a weighted shift operator. 
We assume that  $\lambda_{i}\ne\lambda_{j}$ if $i\ne j$. 
Take a vector  
$\overline{w}=(\overline{w_{n}})_n \in \ell^{2}(\mathbb{N})$ such 
that $w_{n}\ne 0$ for any $n\in \mathbb{N}$.
Put
$A= SD_{\lambda}+\theta_{e_{1},\overline{w}}$ and $B=S$.
Define a Hilbert representation $(H^{\lambda},f^{\lambda})$ of 
the Kronecker quiver $Q$ by 
$H^{\lambda}_{1}=H^{\lambda}_{2}={ H}$, 
$f^{\lambda}_{\alpha}=A$ and $f^{\lambda}_{\beta}=B$. 
Then $\ker A=0$ and the quotient $BA^{-1}$ is a transitive operator.
Furthermore,the operator $BA^{-1}$ is densely defined
if and only if 
$\lambda_{k}\ne 0$ for each $k\in \mathbb{N}$
and $ \displaystyle {(\frac{w_{k}}{\lambda_{k}})_{k}}\not\in \ell^{2}(\mathbb{N})$.
\label{thm:Kron-example}
\end{prop}
\begin{proof}
By \cite[Theorem 3.7.]{EW4}, the Hilbert representation $(H^{\lambda},f^{\lambda})$
is transitive.

For $x=(x_{n})_{n}\in \ell^2(\mathbb{N})$,
assume that  
$$Ax= (SD_{\lambda}+\theta_{e_{1},\overline{w}})x
=(\sum_{n=1}^{\infty}x_{n}w_{n},\lambda_{1}
x_{1},\lambda_{2}x_{2},\cdots)
=0.$$
If $\lambda_{k}\ne0$
for any $k\in \mathbb{N}$, then
$x_{k}=0$ for any $k\in \mathbb{N}$.
If there exists a $k\in \mathbb{N} $ such that
$\lambda_{k}=0$, then
$\lambda_{i}\ne0$ for $i\ne k$.
Hence $x_{i}=0$ for $i\ne k$.
Since $\sum_{n=1}^{\infty}x_{n}w_{n}=x_{k}w_{k}=0$,
$x_{k}=0$ by $w_{k}\ne0$.
Thus we have that $x=0$ and $\ker A=0$.
We note that $\Im B^{*}=\Im S^{*}=H$
and $\Im A^{*}+\Im B^{*}=H$ is closed in $H$.
Hence $BA^{-1}$ is a closed operator.

Next we shall consider the condition such that $BA^{-1}$ is densely defined.
We note that $\overline{D(BA^{-1})}=\overline{\Im A}=(\ker {A^{*}})^{\perp}$.
We shall show that
$\ker A^{*}\ne 0$ if and only if
(1)$\lambda_{k}=0$ for some $k\in \mathbb{N}$
or (2)$\lambda_{k}\ne0$ for any $k\in \mathbb{N}$ and $ \displaystyle {(\frac{w_{k}}{\lambda_{k}})_{k}}\in \ell^{2}(\mathbb{N})$.
We see that $A^{*}=D_{\lambda}^{*}S^{*}+\theta_{\overline{w},e_{1}}$ and 
 $x=(x_{n})_{n}$ is in $\ker A^{*}$
if and only if 
$$(\overline{\lambda_{1}}x_{2},\overline{\lambda_{2}}x_{3},\cdots)
=(-x_{1}\overline{w_{1}},-x_{1}\overline{w_{2}},\cdots).$$
Assume that (1) $\lambda_{k}=0$ for some $k\in \mathbb{N}$.
We put  $x=(x_{i})$ by 
\[
  x_{i} = \begin{cases}
    0 & (i\ne k+1), \\
    1 & (i= k+1).
  \end{cases}
\]
We have that $x\in \ker A^{*}$ and  $ \ker A^{*}\ne 0$. 

Assume that (2) $\lambda_{k}\ne0$ for any $k\in \mathbb{N}$
and $ \displaystyle {(\frac{w_{k}}{\lambda_{k}})_{k}}\in \ell^{2}(\mathbb{N})$.
Take an element
$ \displaystyle x={(1,-\overline{\left(\frac{w_{1}}{\lambda_{1}}\right)},
-\overline{\left(\frac{w_{2}}{\lambda_{2}}\right)},\cdots)}$.
We have $x\in \ker A^{*}$ and  $\ker A^{*}\ne 0$.
Conversely,
assume that there exists $x(\ne0)\in \ker A^{*}$.
Assume that  $x_{1}\ne 0$ .
Since
$$(\overline{\lambda_{1}}x_{2},\overline{\lambda_{2}}x_{3},\cdots)
=(-x_{1}\overline{w_{1}},-x_{1}\overline{w_{2}},\cdots),$$
and $\overline{w_{k}}\ne 0$ for any $k\in  \mathbb{N}$,
we have $\lambda_{k}\ne0$ for any $k\in \mathbb{N}$.

Since $ \displaystyle \left(-\frac{x_{k+1}}{x_{1}}\right)_{k}
\in \ell^{2}(\mathbb{N})$ and $ \displaystyle \left(-\frac{x_{k+1}}{x_{1}}\right)_{k}=
\overline{\left(\frac{w_{k}}{\lambda_{k}}\right)}_{k}$, we have that
$ \displaystyle 
\overline{\left(\frac{w_{k}}{\lambda_{k}}\right)}_{k}\in \ell^{2}(\mathbb{N})$.
Hence we have (2).
Assume that  $x_{1}=0$.
Since $x\ne0$,
there exists $k\in \mathbb{N}$
such that $x_{k+1}\ne0$.
Hence $\lambda_{k}=0$.
Therefore we have (1).
\end{proof}

\noindent{\bf Remark.}
The operator $BA^{-1}$ is densely defined
for $\lambda_{n}=1/n,w_{n}=1/n$ $(n\in \mathbb{N})$. 
The operator $BA^{-1}$ is not densely defined
for  $(\lambda_{n})_{n}$ by 
\[
\lambda_{n} = \begin{cases}
    0 & (n=1), \\
    1/n & (n\ne 1).
  \end{cases}
\]
The operator $BA^{-1}$ is  not densely defined
for
$\lambda_{n} =1-(1/2^{n}),w_{n}=1/n (n\in \mathbb{N})$.

We refer to \cite{Sh} for weighted shifts.

\begin{prop}
Let $Q$ be the  Kronecker quiver and 
${H} = \ell^{2}(\mathbb{Z})$. 
Let $a=(a(n))_{n \in \mathbb{Z}},
b=(b(n))_{n \in \mathbb{Z}}
\in \ell^{\infty}(\mathbb{Z})$
such that
$a(n)\ne0,b(n)\ne0$ for any $n \in \mathbb{Z}.$
We put $\displaystyle w_{m}=\frac{b(m)}{a(m)},m\in 
\mathbb{Z}$. We put
$$\displaystyle 
M_{k}(m,n):=
\frac{w_{m}w_{m+1}\cdots w_{m+k-1}}
{w_{n}w_{n+1}\cdots w_{n+k-1}}
\text{ for }m,n\in \mathbb{Z},k\geq 1.$$
Assume that for any $m\ne n,
(M_{k}(m,n))_k$ is an unbounded sequence.
Let $D_{a}$  be a diagonal operator with 
$a =(a(n))_n$ as diagonal coefficients 
and 
$D_{b}$ be a diagonal operator with 
$b =(b(n))_n$ as diagonal coefficients. 
Let $U$ be the bilateral  forward shift. 
Put $A = D_{a}$ and $B=UD_{b}$.
Define a Hilbert representation $(H,f)$ of 
the Kronecker quiver $Q$ by 
$H_{1}=H_{2}={H}$, 
$f_{\alpha}=A$ and $f_{\beta}=B$. 
Then the Hilbert representation $(H,f)$ 
is transitive. We also have $\ker A=0$ and $\ker B=0$.
And the operator $BA^{-1}$ is a densely defined transitive  operator.
\label{thm:Kro-example}
\end{prop}
\begin{proof}
As in \cite[Theorem 3.8.]{EW4}, we can similarly prove that
the Hilbert representation $(H,f)$ 
is transitive.   By Proposition \ref{unbounded:transitive},
the operator $BA^{-1}$ is transitive.

\end{proof}
\noindent
{\bf Example.}
\cite[Theorem 3.8.]{EW4} 
 Fix a positive constant  $\lambda>1$. 
Consider two sequences 
$a = (a(n))_{n \in \mathbb{Z}}$ and 
$b = (b(n))_{n \in \mathbb{Z}}$ by 
$$
a(n)= 
\begin{cases} e^{-\lambda^{n}}& 
  (n\geq 1, n \text{\ is even }), \\
  1 &  \ \  (otherwise), 
\end{cases}
\ \ \ \ \ 
b(n)=
\begin{cases} e^{-\lambda^{n}}&   (n\geq 1,n \text{\ is odd }),  \\
 1 & \ \  (otherwise).
\end{cases}
$$
These two sequences $a$ and $b$ satisfy the condition of the Proposition.

The concept of transitive operators 
are useful because certain transitive Hilbert 
representations of a quiver are given in terms 
of transitive operators in the next section.

\section{\textbf{Extended Dynkin diagrams 
 and 
transitive Hilbert representations}.}

We consider 
transitive Hilbert representations of quivers 
whose underlying undirected graph is an 
extended Dynkin diagram
 $\widetilde{A_{n}}(n\geq 0)$.
In $\widetilde{A_{0}}$ case, the oriented cyclic quiver is also
called Jordan quiver. 
Trivially we have no infinite-dimensional 
transitive 
Hilbert 
representations of quivers 
whose underlying undirected graph is an 
extended Dynkin diagram
 $\widetilde{A_{0}}$.
Next we consider 
transitive Hilbert representations of quivers 
whose underlying undirected graph is an 
extended Dynkin diagram
 $\widetilde{A_{n}}(n\geq 1)$.
The quiver $C_{n}$ with $n\geq 2$
whose underlying undirected graph is an 
extended Dynkin diagram
 $\widetilde{A_{n-1}}$
is called the oriented cyclic quiver if the quiver has cyclic orientation.
The set $V$ of the vertices of $C_{n}$ is $\{1,2,\cdots,n\}$
and the set $E$ of the arrows of $C_{n}$ is $\{\alpha_{1}, \alpha_{2},\cdots,
\alpha_{n}\}$ such $s(\alpha_{i})=i,r(\alpha_{i})=i+1 
(i=1,\cdots n-1)$ and $s(\alpha_{n})=n,
r(\alpha_{n})=1$.
For  $\tilde{A_{1}}$ case, the quivers are the oriented cyclic quiver $C_{2}$ and 
the Kronecker quiver $Q$.

\begin{thm}
Let $\Gamma $ be a quiver whose underlying undirected graph is
an extended Dynkin diagram $\widetilde{A_{n}},(n\geq 1)$. 
 If 
$\Gamma $ is not an oriented cyclic quiver, then there exists an infinite-dimensional transitive 
Hilbert representation of $\Gamma $.
\label{thm:notcyclic}
\end{thm}
\begin{proof}
Assume that $\Gamma $ is not an oriented cyclic quiver.
Then there exist vertices $i$ and $j$ and arrows
$\alpha$ and $\beta$
such that $s(\alpha)=i,r(\alpha)=i+1$
and $s(\beta)=j+1,r(\beta)=j$$(\mod n)$.
There exists a transitive Hilbert representation $(H,f)$ 
of the Kronecker quiver $Q$ 
given by $A,B\in B(H)$
in \cite[Theorem 3.8.]{EW4}.
We construct a Hilbert representation
 $(H^{\prime},f^{\prime})$ of $\Gamma=(V,E)$ such that
$H^{\prime}_{k}=H(k\in V)$
, $f^{\prime}_{\gamma}=I_{H}$ for $\gamma\ne \alpha,\beta\ (\gamma\in E)$, $f^{\prime}_{\alpha}=A$
and $f^{\prime}_{\beta}=B$.
Then the representation $(H^{\prime},f^{\prime})$ of $\Gamma=(V,E)$
is transitive.

\end{proof}
By Theorem \ref{thm:notcyclic}, the remaining case of the  problem for $\widetilde{A_{n}} (n\geq 1)$
is an oriented cyclic quiver.
It is enough to consider the case that
$H_{i}\ne 0$ for any $i$ by the following lemma.
\begin{lemma}
Let $(H,f)$ be a Hilbert 
representation of the oriented cyclic quiver $C_{n}$.
Assume that there exists a vertex $k$
such that 
$H_{k}=0(1\leq k\leq n)$.
Let  $(K,g)$ be a Hilbert representation
of the oriented cyclic quiver $C_{n-1}$
such that $K_{i}=H_{i}(1\leq i \leq k-1)$,
$K_{i}=H_{i+1}
(k\leq i \leq n-1)$,
$g_{\alpha_{i}}=f_{\alpha_{i}}
(1\leq i \leq k-2)$,
$g_{\alpha_{k-1}}=0,$
$g_{\alpha_{i}}=f_{\alpha_{i+1}}
(k\leq i \leq n-1)$.
Then
$End(H,f)$ is isomorphic to
$End(K,g).$
\label{lemma:End=}
\end{lemma}
\begin{proof}
Take $T=(T_{i})_{i}\in End(H,f)$
for $i=1,\cdots,n$.
Since 
$H_{k}=0$,
$B(H_{k})=0$.
Hence we can associate
$T=(T_{i})_{i}\in End(H,f)$
with
$T^{\prime}=(T^{\prime}_{i})_{i}\in End(K,g)$,
by putting
$T_{i}=T^{\prime}_{i}$
for
$1\leq i\leq (k-1)$
and
$T_{i+1}=T^{\prime}_{i}$
for $k\leq i\leq (n-1)$.
By this correspondence
we have that
$End(H,f)$ is isomorphic to
$End(K,g)$.
\end{proof}

For the case that $H_{i}=\mathbb{C}$ or $0$, we
introduce a concept of an equivalence relation 
for vertices in terms of a Hilbert representation.

\noindent
{\bf Definition.}
Let $(H,f)$ be a Hilbert representation of the oriented cyclic quiver
$C_{n}=(V,E)$ such that $H_{i}=\mathbb{C}$ or $0$.
We give an equivalence relation for the set of vertices 
$\{i\in V;H_{i}\ne 0\}$ as follows:
Take vertices $i,j$
such that $H_{i}\ne 0$ and $H_{j}\ne 0$.
We say that 
vertices $i$ and $j$ are $(H,f)$-connected if
(1) $i=j$ 
or
(2) $i<j$ and $f_{\alpha_{j-1}}\ne0,\cdots, f_{\alpha_{i+1}}\ne0,f_{\alpha_{i}}\ne 0$
or
(3)
$i>j$ and $f_{\alpha_{i-1}}\ne0,\cdots, f_{\alpha_{j+1}}\ne0,f_{\alpha_{j}}\ne 0$.
 
\begin{lemma}
Let $(H,f)$ be a Hilbert representation of the oriented cyclic quiver
$C_{n}$ such that $H_{i}=\mathbb{C}$
or 0 ($i=1,2,\cdots,n$).
Then $(H,f)$ is transitive 
if and only if  
there exists only one
$(H,f)$-connected component.
\label{lemma:component}
\end{lemma}
\begin{proof}
Assume that $(H,f)$ is transitive.
Assume that  there exist two $(H,f)$-connected components $D_{1}$ and
$D_{2}$
in the set
$\{i\in V;H_{i}\ne 0\}$.
Let $\lambda_{1}\in \mathbb{C}$ 
, $\lambda_{2}\in \mathbb{C}$
such that $\lambda_{1}\ne\lambda_{2}$.
We define
$T=(T_{i})_{i\in V}$
by
$T_{i}=\lambda_{1}$ for $i\in D_{1}$
and
$T_{j}=\lambda_{2}$ for $j\in D_{2}$
and
$T_{k}=0$ for $k$(otherwise).
Then $T=(T_{i})_{i\in V}$ is in $End(H,f)$.
This is a contradiction.
Conversely assume that
there exists only one
$(H,f)$-connected component.
Hence there exist decomposition of $V$ by $D_{3}$ and
$D_{4}$ such that
 $$D_{3}\cup D_{4}=V,D_{3}\cap D_{4}=\emptyset,
D_{3}=\{i;H_{i}\ne 0\},
D_{4}=\{j;H_{j}=0\}$$
and $D_{3}$ is the $(H,f)$-connected component.

Let $T=(T_{i})_{i\in V}\in End(H,f).$
Then
$T_{i}=T_{j}$ for $i,j\in D_{3}$.
In fact if $i<j$,
then
$f_{\alpha_{i}}\ne 0,
f_{\alpha_{i+1}}\ne 0,\cdots,
f_{\alpha_{j-1}}\ne 0
$
and
$f_{\alpha_{i}}T_{i}=
T_{i+1}f_{\alpha_{i}},
\cdots
,
f_{\alpha_{j-1}}T_{j-1}=
T_{j}f_{\alpha_{j-1}}$.
Since $f_{\alpha_{i}}\ne 0$ for $i\in D_{3}$,
$T_{i}=T_{i+1}=\cdots=T_{j}$.
Hence
$T_{i}=T_{j}$ for all $i,j\in D_{3}$.
And  $T_{i}=T_{j}=0$ for $i,j\in D_{4}$.
Thus $End(H,f)$ is isomorphic to $\mathbb{C}$.
Hence $(H,f)$ is transitive.

\end{proof}

Next lemma guarantees that we
may assume that
$H_{i}\subset H_{j}$ 
if $\dim H_{i}\leq \dim H_{j}$.
\begin{lemma}
Let $(H_{i})_{i=1}^{n}$
be a family of nonzero Hilbert spaces.
Then there exists a family $(K(i))_{i=1}^{n}$
of subspaces
in a Hilbert space $V$,
such that for any $i(1\leq i\leq n)$,
there exists a number $m(i)(1\leq m(i)\leq n)$
such that $H_{i}$
is isomorphic to $\oplus_{j=1}^{m(i)}K(j)$.
\label{lemma:spacedecomposition}
\end{lemma}
\begin{proof}
We arrange a family of Hilbert spaces  $(H_{i})_{i}$
in increasing order of dimension
and as a result,we have 
$(H_{\ell(1)})$, 
$(H_{\ell(2)})$, 
$\cdots $,
$(H_{\ell(n)})$ in increasing order of dimension.
Construct an ambient space $V$ and its increasing subspaces
$H_{i}^{\prime}\cong H_{i}$ such that 
$(H_{\ell(1)}^{\prime})$ 
$\subset (H_{\ell(2)}^{\prime})$ 
$\subset \cdots $
$\subset (H_{\ell(n)}^{\prime})\subset V$.
Put
$K_{1}=H_{\ell(1)}^{\prime},
K_{2}=H_{\ell(2)}^{\prime}\cap (H_{\ell(1)}^{\prime})^{\perp},
\cdots,
K_{n}=H_{\ell(n)}^{\prime}\cap (H_{\ell(n-1)}^{\prime})^{\perp}.$
Hence there exists a number $m(i)$ such that 
$H_{i}^{\prime}=
 K(1)\oplus K(2)\oplus \cdots \oplus K(m(i))$.
Thus we have that
$H_{i}$
is isomorphic to
$K(1)\oplus K(2)\oplus \cdots \oplus K(m(i))$.
\end{proof}
Firstly we investigate transitive Hilbert representations of oriented cyclic quivers
 $C_{2}$ and $C_{3}$ .
Let $(H,f)$ be a Hilbert representation of $C_{2}$.
In the below we denote  
$f_{\alpha_{1}},f_{\alpha_{2}}$ by $A_{1},A_{2}$ for short.

\begin{lemma}
Let $(H,f)$ be a  transitive Hilbert representation  of $C_{2}$.
Assume that $H_{1}=H_{2}=K\ne 0$,
$A_{1}\in \mathbb{C}$ and $A_{2}\in \mathbb{C}$.
If $A_{1}\ne 0$ or $A_{2}\ne 0$,
then $K=\mathbb{C}$.
\label{lemma 2C} 
\end{lemma}
\begin{proof}
Let $T\in B(K)$.
Then $(T,T)\in End(H,f)$.
In fact $A_{1}T=TA_{1}$
and $A_{2}T=TA_{2}$.
If $\dim K>1$,
$B(K)\ne \mathbb{C}I$.
Since $(H,f)$ is transitive,this is a contradiction.
Thus $\dim K=1$.
\end{proof}

\begin{lemma}
Let $(H,f)$ be a  Hilbert representation of $C_{2}$.
Then $(H,f)$ is transitive if and only if 
one of the following conditions holds.
\begin{enumerate}
\item
$H_1=\mathbb{C},H_{2}=0,A_{1}=0$ and $A_{2}=0,$
\item
 $H_1=0,H_{2}=\mathbb{C}, A_{1}=0$ and \ $A_{2}=0,$
\item
$H_1=\mathbb{C}$ and 
$H_{2}=\mathbb{C}$ \
and ($A_{1}\ne 0$ or
$A_{2}\ne 0$).
\end{enumerate}
\end{lemma}
\begin{proof}
If  (1),(2) or (3) holds, then $(H,f)$ is clearly transitive.
Conversely
assume that $(H,f)$ is transitive.
Assume that $\dim H_{1}\ne 0$ and $\dim H_{2}=0$.
If $\dim H_{1}>1$,
then there exists a non-scalar operator in $B(H_{1})$.
Since $B(H_{1})=End(H,f)$,
this contradicts the transitivity of $(H,f)$. 
Hence $\dim H_{1}=1$. This is the case (1).
Similarly
we have the case (2).
Therefore it is sufficient to assume 
that $\dim H_{1}\ne 0$ and $\dim H_{2}\ne 0$.
By Lemma \ref{lemma:spacedecomposition} we may assume that $\dim H_{1}\leq \dim H_{2}$ and
$H_{1}$ is a subspace of $H_{2}$.
We define 
$$T
=(T_{1},T_{2})
=(A_{2}A_{1},A_{1}A_{2}).$$
Then 
$T\in End(H,f)$.
In fact
$$A_{1}T_{1}=A_{1}(A_{2}A_{1})=
(A_{1}A_{2})A_{1}=T_{2}A_{1}$$
and
$$T_{1}A_{2}=(A_{2}A_{1})A_{2}=
A_{2}(A_{1}A_{2})=A_{2}T_{2}.$$
By the assumption of transitivity for $(H,f)$,
$$(T_{1},T_{2})\in \{(\mu I_{H_{1}},
\mu I_{H_{2}})\vert \mu\in \mathbb{C}\}.$$
Hence
$$T_{1}=A_{2}A_{1}=\mu I_{H_{1}}
,T_{2}=A_{1}A_{2}=\mu I_{H_{2}}
\text{ for some } \mu\in \mathbb{C}.$$
We denote by $E_{1}\in B(H_{1},H_{2})$
the embedding map of $H_{1}$ into $H_{2}$
and
$E_{2}\in B(H_{2},H_{1})$ 
the projection map of $H_{2}$ onto $H_{1}$.
We define
$$T^{\{1\}}
=(T^{\{1\}}_{1},T^{\{1\}}_{2})
=(A_{2}E_{1},E_{1}A_{2}).$$
Then 
$T^{\{1\}}\in End(H,f)$.
In fact
\begin{align*}
A_{1}T_{1}^{\{1\}}
&=
A_{1}(A_{2}E_{1})=(A_{1}A_{2})E_{1}=
\mu I_{H_{2}}E_{1}
\\
&=\mu E_{1}=
E_{1}\mu I_{H_{1}}
=(E_{1}A_{2})A_{1}
=T^{\{1\}}_{2}A_{1}.
\end{align*}

$$T_{1}^{\{1\}}A_{2}=
(A_{2}E_{1})A_{2}=A_{2}(E_{1}A_{2})
=A_{2}T^{\{1\}}_{2}.$$
Thus $T^{\{1\}}\in End(H,f)$.
Since $(H,f)$ is transitive,
there exists a constant
$\mu^{\{1\}}\in \mathbb{C}$
such that
$$A_{2}E_{1}=\mu^{\{1\}}I_{H_{1}}
\text { and }
E_{1}A_{2}=\mu^{\{1\}}I_{H_{2}}.$$
We define $$T^{\{2\}}
=(T^{\{2\}}_{1},T^{\{2\}}_{2})
=(E_{2}A_{1},A_{1}E_{2}).$$
Then 
$T^{\{2\}}\in End(H,f)$.
In fact,
$$A_{1}T_{1}^{\{2\}}=A_{1}(E_{2}A_{1})
=(A_{1}E_{2})A_{1}=T^{\{2\}}_{2}A_{1}.$$
\begin{align*}
T_{1}^{\{2\}}A_{2}
&=(E_{2}A_{1})A_{2}
=E_{2}(\mu I_{H_{2}})
=\mu E_{2}\\
&=\mu I_{H_{1}}E_{2}
=A_{2}(A_{1}E_{2})=A_{2}T^{\{2\}}_{2}.
\end{align*}

Since $(H,f)$ is transitive,
there exists a constant
$\mu^{\{2\}}\in \mathbb{C}$
such that
$E_{2}A_{1}=\mu^{\{2\}}I_{H_{1}}$ and
$A_{1}E_{2}=\mu^{\{2\}}I_{H_{2}}$. 

We define
 $$T^{\{1,2\}}
=(T^{\{1,2\}}_{1},T^{\{1,2\}}_{2})
=(E_{2}E_{1},E_{1}E_{2}).$$
Then 
$T^{\{1,2\}}\in End(H,f)$.
In fact,
\begin{align*}
&A_{1}T_{1}^{\{1,2\}}
=
A_{1}(E_{2}E_{1})
=(A_{1}E_{2})E_{1}
=\mu^{\{2\}}I_{H_{2}}E_{1}
=\mu^{\{2\}}E_{1}\\
&=\mu^{\{2\}}E_{1}=E_{1}(\mu^{\{2\}}I_{H_{1}})
=E_{1}(E_{2}A_{1})
=
T^{\{1,2\}}_{2}A_{1},\\
&T_{1}^{\{1,2\}}A_{2}
=(E_{2}E_{1})A_{2}
=
E_{2}(\mu^{\{1\}}I_{H_{2}})
=\mu^{\{1\}}E_{2}
\\
&=\mu^{\{1\}}E_{2}=(\mu^{\{1\}}I_{H_{1}})E_{2}
=A_{2}(E_{1}E_{2})
=A_{2}T^{\{1,2\}}_{2}.
\end{align*}
Since $(H,f)$ is transitive,
there exists a constant
$\mu^{\{1,2\}}\in \mathbb{C}$
such that
$$E_{2}E_{1}=\mu^{\{1,2\}}I_{H_{1}}
\text{ and }
E_{1}E_{2}=\mu^{\{1,2\}}I_{H_{2}}.$$
For $x(\ne 0)\in H_{1}$, we have  $x=E_{2}E_{1}x=\mu^{\{1,2\}}I_{H_{1}}x=\mu^{\{1,2\}}x$.
Hence $\mu^{\{1,2\}}=1$.
If $H_{1}\ne H_{2}$,then $H_{1}^{\perp}\cap H_{2}\ne 0$.
Take $x(\ne 0)\in H_{1}^{\perp}\cap H_{2}$.
Then $E_{1}E_{2}x=\mu^{\{1,2\}}I_{H_{2}}x$.
Hence $0=x$. This is a contradiction.
Thus $H_{1}=H_{2}$ and 
$E_{1}=E_{2}$.
Since 
$A_{1}E_{2}=\mu^{\{2\}}I_{H_{2}}$,
$A_{1}=\mu^{\{2\}}I_{H_{1}}$.
And we also have 
$E_{1}A_{2}=A_{2}=\mu^{\{1\}}I_{H_{1}}$.
Since  
$(H,f)$ is transitive,
$A_{1}\ne0$ or $A_{2}\ne0$.
By Lemma \ref{lemma 2C}, we have  $H_{1}=H_{2}=\mathbb{C}$.
Thus $(H,f)$ is  in the case (3).
 
\end{proof}

Let $(H,f)$ be a Hilbert representation of the oriented cyclic quiver $C_{3}$.
In the below we denote 
$f_{\alpha_{1}}$,
$f_{\alpha_{2}}$,
$f_{\alpha_{3}}$ 
by $A_{1},A_{2},A_{3}$ for short.

\begin{lemma}
Let $(H,f)$ be a transitive Hilbert representation of $C_{3}$.
Assume that $H_{i}=\mathbb{C}(i=1,2,3)$ .
Then $A_{i}A_{j}\ne 0$ for some $i\ne j$.
\label{lemma-component}
\end{lemma}
\begin{proof}
Assume that $A_{i}=A_{j}=0$ for some $i\ne j$.
We may and do assume $i=1,j=2$.
Let $T=(T_{1},T_{2},T_{3})$
such that $T_{1}=T_{3}$,$T_{2}\ne T_{1}$,
$T_{1}\ne 0$ and $T_{2}\ne 0$.
Then $T=(T_{1},T_{2},T_{3})$ is in $End(H,f)$.
Since $(H,f)$ is transitive,
 $T_{1}=T_{2}=T_{3}\in \mathbb{C}$.
This is a contradiction.
Hence this lemma holds.
\end{proof}


\begin{lemma}

Let $(H,f)$ be a Hilbert representation of $C_{3}$.
Then $(H,f)$ is transitive if and only if 
one of the following holds.

\begin{enumerate}

\item  
$H_{1}=
\mathbb{C}\text{ and }H_{i}=0(i=2,3)
$.
\item
$H_{2}=\mathbb{C}\text{ and }H_{i}=0(i=1,3)$.
\item
$H_{3}=
\mathbb{C}\text{ and }H_{i}=0(i=1,2)$.
\item
$H_{i}=\mathbb{C}(i=1,2),$$ H_{3}=0$ and
$A_{1}\ne 0
$.
\item
$H_{i}=\mathbb{C}(i=2,3),$$H_{1}=0$ and
$A_{2}\ne 0$.
\item
$H_{i}=\mathbb{C}(i=1,3),$$ H_{2}=0$ and
$A_{3}\ne 0$.
\item
$H_{i}=\mathbb{C}(i=1,2,3)$ and
$A_{i}A_{j}\ne 0$ for some $i\ne j(i,j=1,2,3)$.

\end{enumerate}

\end{lemma}
\begin{proof}

If a Hilbert representations $(H,f)$ satisfies 
(1),(2),$\cdots$ or (7), then 
the Hilbert representation is obviously transitive.
Conversely assume that $(H,f)$ is transitive.
At first we assume that 
all Hilbert spaces $H_{i}\ne 0 (1\leq i\leq 3)$
and by Lemma \ref{lemma:spacedecomposition}
a totally ordered set by inclusion order and
$H_{1}\subset H_{i}(i=2,3)$.
We define
$$ T_{1}= A_{3}A_{2}A_{1},
T_{2}= A_{1}A_{3}A_{2},
T_{3}= A_{2}A_{1}A_{3}\text{ and }
T=
(T_{1},
T_{2},T_{3}).$$
We define a mapping 
$E_{i}\in B(H_{i},H_{i+1})$ by
$$
E_{i}= 
\begin{cases}\text{ the inclusion map of } H_{i} \text{ into } H_{i+1}
&  \  \text{if }  H_{i}\subset  H_{i+1},  \\
 \text{ the projection map of } H_{i} \text{ onto } H_{i+1} 
&  \   \text{ if }H_{i+1}\subset H_{i}.
\end{cases}
$$
For a subset $S$ of $\{1,2,3\},$
we define $B_{i}\in B(H_{i},H_{i+1})$ by

$$
B_{i}= 
\begin{cases} A_{i}&  \ \ \text{ if }i\notin S,  \\
  E_{i}&  \ \  \text{ if }i\in S.
\end{cases}
$$
We also define
$$T_{1}^{S}=B_{3}B_{2} B_{1},
T_{2}^{S}=B_{1}B_{3} B_{2},
T_{3}^{S}=B_{2}B_{1} B_{3}
\text{ and } 
T^{S}=(T_{1}^{S},T_{2}^{S},T_{3}^{S}).$$
We note that $T^{S}=(T_{1}^{S},T_{2}^{S},T_{3}^{S})$ 
is obtained by replacing each word $A_{i}$
in $T=(T_{1},T_{2},T_{3})$
with
$E_{i}$ for all $i\in S$.
We regard $T=(T_{1},T_{2},T_{3})$
as $T^{\emptyset}=
(T^{\emptyset}_{1},T^{\emptyset}_{2},T^{\emptyset}_{3})$.
Since
$$A_{1}T_{1}
=A_{1}(A_{3}A_{2}A_{1})
=T_{2}A_{1},$$
$$A_{2}T_{2}
=A_{2}(A_{1}A_{3}A_{2})
=T_{3}A_{2},$$
$$A_{3}T_{3}
=A_{3}(A_{2}A_{1}A_{3})
=
T_{1}A_{3},$$
we have that
$T$ is in   $End(H,f)$.
Since $(H,f)$ is transitive,
there exists a constant $\mu\in \mathbb{C}$
such that 
$$A_{3}A_{2}A_{1}=\mu I_{H_{1}},
A_{1}A_{3}A_{2}=\mu I_{H_{2}},
A_{2}A_{1}A_{3}=\mu I_{H_{3}}.$$
For $S=\{1\}$, we define
$T^{S}=T^{\{1\}}=(T^{\{1\}}_{1},T^{\{1\}}_{2},T^{\{1\}}_{3})$ by
$$T^{\{1\}}_{1}= A_{3}A_{2}E_{1},
T^{\{1\}}_{2}= E_{1}A_{3}A_{2},
T^{\{1\}}_{3}= A_{2}E_{1}A_{3}.$$
It follows that
\begin{align*}
&A_{1}T_{1}^{\{1\}}
=A_{1}A_{3}A_{2}E_{1}
=\mu I_{H_{2}}E_{1}=\mu E_{1} \\
&=E_{1}\mu I_{H_{1}}
=E_{1}(A_{3}A_{2})A_{1}=T_{2}^{\{1\}}A_{1},\\
&A_{2}T_{2}^{\{1\}}=A_{2}E_{1}(A_{3}A_{2})
=T_{3}^{\{1\}}A_{2},\\
&A_{3}T_{3}^{\{1\}}=A_{3}A_{2}E_{1}A_{3}
=T_{1}^{\{1\}}A_{3}.
\end{align*}
Thus $T^{\{1\}}$ is in $ End(H,f)$.
Since $(H,f)$ is transitive,
there exists a constant $\mu^{\{1\}}\in \mathbb{C}$
such that
$$A_{3}A_{2}E_{1}=\mu^{\{1\}} I_{H_{1}},
E_{1}A_{3}A_{2}=\mu^{\{1\}} I_{H_{2}},
,A_{2}E_{1}A_{3}=\mu^{\{1\}} I_{H_{3}}.$$
For $S=\{2\}$,
we define
$T^{S}=T^{\{2\}}=(T^{\{2\}}_{1},T^{\{2\}}_{2},T^{\{2\}}_{3})$ by
$$T^{\{2\}}_{1}= A_{3}E_{2}A_{1},
T^{\{2\}}_{2}= A_{1}A_{3}E_{2},
T^{\{2\}}_{3}= E_{2}A_{1}A_{3}.$$
It follows that
$$A_{1}T_{1}^{\{2\}}=A_{1}A_{3}E_{2}A_{1}
=T_{2}^{\{2\}}A_{1},$$

\begin{align*}
A_{2}T_{2}^{\{2\}}
&=A_{2}A_{1}A_{3}E_{2}
=\mu I_{H_{3}} E_{2}
=\mu E_{2}\\
&=E_{2}\mu I_{H_{2}}
=E_{2}A_{1}A_{3}A_{2}=T_{3}^{\{2\}}A_{2},
\end{align*}
$$A_{3}T_{3}^{\{2\}}=A_{3}E_{2}A_{1}A_{3}
=T_{1}^{\{2\}}A_{3}.$$
Thus 
$T^{\{2\}}$ is in  $End(H,f)$.
Since $(H,f)$ is transitive,
there exists a constant
$\mu^{\{2\}}\in \mathbb{C}$
such that
$$A_{3}E_{2}A_{1}=\mu^{\{2\}} I_{H_{1}},
A_{1}A_{3}E_{2}=\mu^{\{2\}} I_{H_{2}},
,(E_{2}A_{1})A_{3}=\mu^{\{2\}} I_{H_{3}}.$$
For
$S=\{3\}$, we define
$T^{S}=T^{\{3\}}=(T^{\{3\}}_{1},T^{\{3\}}_{2},T^{\{3\}}_{3})$ by
$$T^{\{3\}}_{1}= E_{3}A_{2}A_{1},
T^{\{3\}}_{2}= A_{1}E_{3}A_{2},
T^{\{3\}}_{3}= A_{2}A_{1}E_{3}.$$
It follows that
$$A_{1}T_{1}^{\{3\}}
=A_{1}E_{3}A_{2}A_{1}
=T_{2}^{\{3\}}A_{1},$$
$$A_{2}T_{2}^{\{3\}}
=A_{2}A_{1}E_{3}A_{2}
=T_{3}^{\{3\}}A_{2},$$

\begin{align*}
A_{3}T_{3}^{\{3\}}
&=A_{3}A_{2}A_{1}E_{3}
=\mu I_{H_{1}}E_{3}
=\mu E_{3} \\
&=E_{3}\mu I_{H_{3}}
=E_{3}A_{2}A_{1}A_{3}
=T_{1}^{\{3\}}A_{3}.
\end{align*}
Thus 
$T^{\{3\}}$ is in  $End(H,f)$.
Since $(H,f)$ is transitive,
there exists a constant
$\mu^{\{3\}}\in \mathbb{C}$
such that
$$E_{3}A_{2}A_{1}=\mu^{\{3\}} I_{H_{1}},
A_{1}E_{3}A_{2}=\mu^{\{3\}} I_{H_{2}},
,A_{2}A_{1}E_{3}=\mu^{\{3\}} I_{H_{3}}.$$
For $S=\{1,2\}$, we have
$$T^{\{1,2\}}=(T^{\{1,2\}}_{1},
T^{\{1,2\}}_{2},
T^{\{1,2\}}_{3})
=(A_{3}E_{2}E_{1},
E_{1}A_{3}E_{2},
E_{2}E_{1}A_{3}).$$
It follows that 
\begin{align*}
A_{1}T_{1}^{\{1,2\}}
&=A_{1}A_{3}E_{2}E_{1}
=\mu^{\{2\}}I_{H_{2}}E_{1}=\mu^{\{2\}}E_{1}\\
&=E_{1}\mu^{\{2\}}I_{H_{1}}
=E_{1}A_{3}E_{2}A_{1}
=T_{2}^{\{1,2\}}A_{1}
,
\end{align*}
\begin{align*}
A_{2}T_{2}^{\{1,2\}}
&=A_{2}E_{1}A_{3}E_{2}
=\mu^{\{1\}}I_{H_{3}}E_{2}
=\mu^{\{1\}}E_{2}\\
&=E_{2}\mu^{\{1\}}I_{H_{2}}
=E_{2}E_{1}A_{3}A_{2}
=T_{3}^{\{1,2\}}A_{2}
,
\end{align*}

$$A_{3}T_{3}^{\{1,2\}}
=A_{3}E_{2}E_{1}A_{3}
=T_{1}^{\{1,2\}}A_{3}
.$$
Thus 
$T^{\{1,2\}}$ is in  $End(H,f)$.
Since $(H,f)$ is transitive,
there exists a constant
$\mu^{\{1,2\}}\in \mathbb{C}$
such that
$$A_{3}E_{2}E_{1}=\mu^{\{1,2\}}I_{H_{1}},
E_{1}A_{3}E_{2}
=\mu^{\{1,2\}}I_{H_{2}},
E_{2}E_{1}A_{3}
=\mu^{\{1,2\}}I_{H_{3}}.$$
For $S=\{1,3\}$, we have
$$T^{\{1,3\}}=(T^{\{1,3\}}_{1},T^{\{1,3\}}_{2},T^{\{1,3\}}_{3})=
(E_{3}A_{2}E_{1},
E_{1}E_{3}A_{2},
A_{2}E_{1}E_{3}).$$
It follows that
\begin{align*}
A_{1}T_{1}^{\{1,3\}}
&=A_{1}E_{3}A_{2}E_{1}
=\mu^{\{3\}}I_{H_{2}}E_{1}
=\mu^{\{3\}}E_{1}\\
&=E_{1}\mu^{\{3\}}I_{H_{1}}
=E_{1}E_{3}A_{2}A_{1}
=T_{2}^{\{1,3\}}A_{1}.
\end{align*}
$$A_{2}T_{2}^{\{1,3\}}
=A_{2}E_{1}E_{3}A_{2}
=T_{3}^{\{1,3\}}A_{2}
,$$
\begin{align*}
A_{3}T_{3}^{\{1,3\}}
&=A_{3}A_{2}E_{1}E_{3}
=\mu^{\{1\}}I_{H_{1}}E_{3}
=\mu^{\{1\}}E_{3}\\
&=E_{3}\mu^{\{1\}}I_{H_{3}}
=E_{3}A_{2}E_{1}A_{3}
=T_{1}^{\{1,3\}}A_{3}
.
\end{align*}
Thus 
$T^{\{1,3\}}$ is in $End(H,f)$.
Since $(H,f)$ is transitive,
there exists a constant
$\mu^{\{1,3\}}\in \mathbb{C}$
such that
$$E_{3}A_{2}E_{1}=\mu^{\{1,3\}}I_{H_{1}}
,
E_{1}E_{3}A_{2}
=\mu^{\{1,3\}}I_{H_{2}},
A_{2}E_{1}E_{3}
=\mu^{\{1,3\}}I_{H_{3}}
.$$
For $S=\{2,3\}$,we have
$$T^{\{2,3\}}=
(T^{\{2,3\}}_{1},T^{\{2,3\}}_{2},T^{\{2,3\}}_{3})
=(E_{3}E_{2}A_{1},A_{1}E_{3}E_{2},E_{2}A_{1}E_{3}).$$
It follows that
$$A_{1}T_{1}^{\{2,3\}}
=A_{1}E_{3}E_{2}A_{1}
=T_{2}^{\{2,3\}}A_{1}.$$
\begin{align*}
A_{2}T_{2}^{\{2,3\}}
&=A_{2}A_{1}E_{3}E_{2}
=\mu^{\{3\}}I_{H_{3}}E_{2}=\mu^{\{3\}}E_{2}\\
&=E_{2}\mu^{\{3\}}I_{H_{2}}
=E_{2}A_{1}E_{3}A_{2}
=T_{3}^{\{2,3\}}A_{2}
,
\end{align*}
\begin{align*}
A_{3}T_{3}^{\{2,3\}}
&=A_{3}E_{2}A_{1}E_{3}
=\mu^{\{2\}}I_{H_{1}}E_{3}=\mu^{\{2\}}E_{3}\\
&=E_{3}\mu^{\{2\}}I_{H_{3}}
=E_{3}E_{2}A_{1}A_{3}
=T_{1}^{\{2,3\}}A_{3}
.
\end{align*}
Thus 
$T^{\{2,3\}}$ is in $End(H,f)$.
Since $(H,f)$ is transitive,
there exists a constant
$\mu^{\{2,3\}}\in \mathbb{C}$
such that
$$E_{3}E_{2}A_{1}=\mu^{\{2,3\}}I_{H_{1}}
,
A_{1}E_{3}E_{2}
=\mu^{\{2,3\}}I_{H_{2}},
E_{2}A_{1}E_{3}
=\mu^{\{2,3\}}I_{H_{3}}
.$$
For $S=\{1,2,3\}$,we have
$$T^{\{1,2,3\}}=
(T^{\{1,2,3\}}_{1},T^{\{1,2,3\}}_{2},T^{\{1,2,3\}}_{3})
=(E_{3}E_{2}E_{1},E_{1}E_{3}E_{2},E_{2}E_{1}E_{3}).$$
It follows that
\begin{align*}
A_{1}T_{1}^{\{1,2,3\}}
&=A_{1}E_{3}E_{2}E_{1}
=\mu^{\{2,3\}}I_{H_{1}}E_{1}=\mu^{\{2,3\}}E_{1}\\
&=E_{1}\mu^{\{2,3\}}I_{H_{1}}
=E_{1}E_{3}E_{2}A_{1}
=T_{2}^{\{1,2,3\}}A_{1}.
\end{align*}
\begin{align*}
A_{2}T_{2}^{\{1,2,3\}}
&=A_{2}E_{1}E_{3}E_{2}
=\mu^{\{1,3\}}I_{H_{3}}E_{2}=\mu^{\{1,3\}}E_{2}\\
&=E_{2}\mu^{\{1,3\}}I_{H_{2}}
=E_{2}E_{1}E_{3}A_{2}
=T_{3}^{\{1,2,3\}}A_{2}
,
\end{align*}
\begin{align*}
A_{3}T_{3}^{\{1,2,3\}}
&=A_{3}E_{2}E_{1}E_{3}
=\mu^{\{1,2\}}I_{H_{1}}E_{3}=\mu^{\{1,2\}}E_{3}\\
&=E_{3}\mu^{\{1,2\}}I_{H_{3}}
=E_{3}E_{2}E_{1}A_{3}
=T_{1}^{\{1,2,3\}}A_{3}
.
\end{align*}
Thus 
$T^{\{1,2,3\}}$ is in  $End(H,f)$.
Since $(H,f)$ is transitive,
there exists a constant
$\mu^{\{1,2,3\}}\in \mathbb{C}$
such that
$$E_{3}E_{2}E_{1}=\mu^{\{1,2,3\}}I_{H_{1}}
,
E_{1}E_{3}E_{2}
=\mu^{\{1,2,3\}}I_{H_{2}},
E_{2}E_{1}E_{3}
=\mu^{\{1,2,3\}}I_{H_{3}}
.$$
Take $x(\ne 0)\in H_{1}$.
Since $H_{1}\subset H_{i}(1\leq i\leq 3)$,

$E_{3}E_{2}E_{1}x=x=\mu^{\{1,2,3\}}I_{H_{1}}x
.$
 Hence $\mu^{\{1,2,3\}}=1$.
By Lemma \ref{lemma:spacedecomposition},
we can represent $H_{i}$ by
$H_{i}=K_{1}\oplus K_{2}\oplus \cdots \oplus K_{m(i)}(1\leq i\leq 3)$.
We shall show that
$$H_{1}=H_{2}=H_{3}.$$
Now, $m(1)=1$ and
assume that $m(2)\ne 1$.
We compare $m(3)$ with $m(2)$.
Assume that
$m(3)<m(2)$.
Take $x(\ne 0)\in K_{m(2)}\subset H_{2}$.
Then
$$E_{2}x=0.$$
This contradicts that
$$E_{1}E_{3}E_{2}
=I_{H_{2}}.$$
Assume that
$m(3)\geq m(2)$.
Take $x(\ne 0)\in K_{m(2)}\subset H_{2}$.
Then
$$E_{2}x=x 
\text{  and  }
E_{3}x=0.$$
This contradicts that
$$E_{1}E_{3}E_{2}
=I_{H_{2}}.$$

Hence $m(1)=m(2)$ and $H_{1}=H_{2}$.
Next assume that
$H_{3}\ne H_{1}$ (hence $m(3)\ne 1$).
Take $x(\ne 0)\in K_{m(3)}\subset H_{3}$.
Then
$$E_{3}x=0.$$
This contradicts that
$$E_{2}E_{1}E_{3}
=I_{H_{3}}.$$
Hence we have 
that
$$H_{1}=H_{2}=H_{3}:=M.$$
Therefore
$$E_{1}=E_{2}=E_{3}=I_{M} \text{ and }
T_{i}^{\{1,2,3\}}=I_{M}\in \mathbb{C}.$$
Since 
$$E_{3}E_{2}A_{1}
=\mu^{\{2,3\}} I_{H_{1}},
E_{1}E_{3}A_{2}
=\mu^{\{1,3\}} I_{H_{2}},
\text{ and }
E_{2}E_{1}A_{3}
=\mu^{\{1,2\}}I_{H_{3}},$$
$$A_{1}=\mu^{\{2,3\}} I_{M},
A_{2}=\mu^{\{1,3\}} I_{M},
\text{ and }
A_{3}=\mu^{\{1,2\}} I_{M}.$$
If $\dim M>1$,
there is a non-scalar operator $B\in B(M)$.
Since $A_{1},A_{2},A_{3}$ are scalar operators,
$(B,B,B)\in End(H,f)$.
This contradicts that $(H,f)$ is transitive.
Hence we have $\dim M=1$.
By Lemma \ref{lemma-component},$A_{i}A_{j}\ne 0$ for some $i\ne j(i,j=1,2,3)$.
Thus $(H,f)$ is in the case (7).

Next we consider other cases.
Assume that there exists  $H_{i}= 0$ for some $i$.
Since $(H,f)$ is transitive, the number $\vert\{i;H_{i}\ne0\}\vert$
is 1 or 2.
If  $\vert\{i;H_{i}\ne0\}\vert=1=\vert\{k\}\vert$,then
$\dim H_{k}=1$ because $(H,f)$ is transitive.
Hence these are in the cases (1),(2),(3).
If  $\vert\{i;H_{i}\ne0\}\vert=2=\vert\{k,\ell\}\vert, (k<\ell \mod 3)$,
then we consider the reduction $C_{2}$ of the quiver $C_{3}$
as it is shown in Lemma
\ref{lemma:End=}.
Let $(K,g)$ be the reduced Hilbert representation of $C_{2}$
from the Hilbert representation $(H,f)$ of $C_{2}$ by Lemma
\ref{lemma:End=}.
We have $End(H,f)\cong End(K,g)$.
Hence $End(K,g)$ is transitive.
By the same argument in the case (7),
we have $\dim H_{k}=\dim H_{\ell}=1$.
Since $(H,f)$ is transitive, $A_{k}\ne0$.
Thus these are in the cases (4),(5),(6).
All these cases are  summarized  as the existence of 
unique $(H,f)$-connected component by Lemma 
\ref{lemma:component}.
\end{proof}
Let $(H,f)$ be a Hilbert representation of $C_{n}$.
In the below we denote 
$f_{\alpha_{1}},
f_{\alpha_{2}},\cdots,
f_{\alpha_{n}}$ by $A_{1},A_{2},\cdots,A_{n}$
for short.

\begin{lemma}
Let $(H,f)$
be a Hilbert representation of the oriented cyclic quiver $C_{n}$.
Then $(H,f)$ is transitive if and only if 
$H_{i}=\mathbb{C}$ or 0 and there  exists  only one $(H,f)$-connected 
component in $\{i \in V ; H_i \not= 0\}$. 
\label{lemma:cyclictransitive}
\end{lemma}
\begin{proof}
Assume that 
$H_{i}=\mathbb{C}$ or 0 and there  exists  only one $(H,f)$-connected 
component in $\{i \in V ; H_i \not= 0\}$.
Then $(H,f)$ is transitive by Lemma \ref{lemma:component}. 

Conversely assume that $(H,f)$ is transitive.
At first we consider the case that $H_{i}\ne 0$ for any i.
By lemma \ref{lemma:spacedecomposition}, 
we may  and do assume that the family $(H_{i})$
of Hilbert spaces are totally ordered under the inclusion order . 
We also assume that $\dim H_{1}$
is the smallest dimension among $\{\dim H_{i};i=1,\cdots,n\}$.

We define
$T=
(T_{1},
T_{2},\cdots,T_{n})$ by
$$T_{1}=A_{n} \cdots A_{3}A_{2}A_{1},
T_{2}= A_{1}A_{n}\cdots A_{4}A_{3}A_{2},
\cdots,
T_{n}= A_{n-1}\cdots A_{3}A_{2}A_{1}A_{n}.$$
Then $T=(T_{1},T_{2},\cdots, T_{n})$
is clearly in $End(H,f)$.


We denote by $E_{i}$ 
the following operator $E_{i}: H_i \rightarrow H_{i+1}$:   
$$
E_{i}= 
\begin{cases}\text{ the inclusion map from } H_{i} \text{ into } H_{i+1}
&  \ \ (\text{if }  H_{i}\subset  H_{i+1} ), \\
\text{ the projection map from} H_{i} \text{ onto } H_{i+1} 
&  \ \ (\text{ if }H_{i+1}\subset H_{i}).
\end{cases}
$$

For $S\subseteq \{1,2,\cdots ,n\}$,
we define 
$B_{i}\in B(H_{i},H_{i+1})$, which depends on $S$,  by

$$
B_{i}= 
\begin{cases} A_{i},&  \ \ \text{ if }i\notin S  \\
 E_{i}, &  \ \  \text{ if }i\in S
\end{cases}
$$

We also define $T_{i}^{S} \in B(H_i)$ and $T^{S} 
\in B(H_1 \oplus ... \oplus H_n)$ by  
$$T_{i}^{S}=B_{i-1}B_{i-2}\cdots B_{2}B_{1}B_{n}B_{n-1}\cdots B_{i+1}B_{
i}
\ \ \ \text{ for } 1\leq i\leq n ,$$
and $T^{S}=(T_{1}^{S},T_{2}^{S},\cdots, T_{n}^{S})$.
That is, $T^{S}=(T_{1}^{S},T_{2}^{S},\cdots, T_{n}^{S})$
is obtained by replacing each word $A_{i}$
in $T=(T_{1},T_{2},\cdots, T_{n})$
with
$E_{i}$ for all $i\in S$.

For example, $T^{\{1\}}=(T^{\{1\}}_{1},
T^{\{1\}}_{2},$$\cdots, T^{\{1\}}_{n})$ 
is given by 
$$T^{\{1\}}_{1}=A_{n}A_{n-1}\cdots A_{2}E_{1},T^{\{1\}}_{2}=
E_{1}A_{n}A_{n-1}\cdots A_{3}A_{2},$$
$$T^{\{1\}}_{3}=
A_{2}E_{1}A_{n}A_{n-1}\cdots A_{4}A_{3},
\cdots ,T^{\{1\}}_{n}=
A_{n-1}A_{n-2}\cdots A_{2}E_{1}A_{n}.$$
We regard $T$ as $T^{\emptyset}$.

In the following we shall show that 
$T^{S}=(T_{1}^{S},T_{2}^{S},\cdots, T_{n}^{S})$
is in $End(H,f)$ for any $S \subseteq \{1,2,\cdots ,n\}$.
We shall prove it by the induction  on the number $k = \vert S\vert$.
First consider the case $k = \vert S\vert = 0$, that is, 
$S = \emptyset$. Then $T^{\emptyset} = T=(T_{1},T_{2},\cdots, T_{n})$
is clearly in $End(H,f)$.

Next, we assume that $T^S$ is in $End(H,f)$ for $\vert S\vert=k$.
Since $(H,f)$ is transitive, there exists a constant 
${\mu}^S \in \mathbb{C}$ such that $T_i^{S} = {\mu}^S I_{H_i}$ 
for any $i = 1,\dots,n$. 
Take $S$ such that $\vert S\vert=k+1$.
We shall show that $T^S$ is in $End(H,f)$. 
It is enough to show that, for any $i = 1,\dots,n$,  
$$
A_iT_i^{S} = T_{i+1}^SA_i
$$
First we consider the case that $i = 1$. 
We need to show  the validity of the relation
$$
A_1T_1^{S} = T_2^SA_1, \ \ \text{that is,} \ 
A_{1}B_{n}\cdots B_{2}B_{1}=B_{1}B_{n}\cdots B_{2}A_{1}.
$$
Assume that  $1$ is in $S$. 
Then $B_{1}=E_{1}$ and $T_i^{S \setminus \{1\}}$ is in $End(H,f)$ 
by the assumption of the induction.  
Since 
$A_{1}B_{n}B_{n-1}\cdots B_{2}$ and 
$B_{n}B_{n-1}\cdots B_{2}A_{1}$ have $k$ changed letters, 
we have
$$T_2^{S \setminus \{1\}} = A_{1}B_{n}B_{n-1}\cdots B_{2}
=\mu^{S\setminus\{1\}}I_{H_{2}}$$
and
$$T_1^{S \setminus \{1\}}= B_{n}B_{n-1}\cdots B_{2}A_{1}
=\mu^{S\setminus\{1\}}I_{H_{1}}.$$
Therefore we have 
$$A_1T_1^{S} = A_{1}B_{n}\cdots B_{1}=\mu^{S\setminus\{1\}}I_{H_{2}}B_{1}
=\mu^{S\setminus\{1\}}E_{1}$$
and
$$T_2^SA_1 = B_{1}B_{n}B_{n-1}\cdots B_{2}A_{1}
=B_{1}\mu^{S\setminus\{1\}}I_{H_{1}}=\mu^{S\setminus\{1\}}E_{1}.$$
Thus
$A_1T_1^{S} = T_2^SA_1$.

Assume that  $1$ is not in $S$. 
Then $B_{1}=A_{1}$.  Hence 
$$
A_1T_1^{S}
= A_{1}B_{n}\cdots B_{1}=A_{1}B_{n}\cdots B_{2}A_{1}$$
and
$$
T_2^SA_1
= B_{1}B_{n}\cdots B_{2}A_{1}=A_{1}B_{n}\cdots B_{2}A_{1}.$$
Thus $A_1T_1^{S} = T_2^SA_1$. 

For other cases that $i = 2,3,\dots n$, we also have that 
$$
A_iT_i^{S} = T_{i+1}^SA_i
$$ 
Hence, by induction, 
we have that 
$T^S$ is in $End(H,f)$ for any $S\subset \{1,2,\cdots n\}$.

In particular, put $S = \{1,2,\cdots, n\}$. 
Since 
$$
T^{\{1,2,\cdots, n\}}=(T_{1}^{\{1,2,\cdots, n\}}
,T_{2}^{\{1,2,\cdots, n\}},
\cdots
,T_{n}^{\{1,2,\cdots, n\}})
$$
is in $End(H,f)$ and $(H,f)$ is transitive, there exits a 
constant $\mu^{\{1,2,\cdots, n\}} \in \mathbb{C}$ such that 
$$
T_{i}^{\{1,2,\cdots, n\}}
= E_{i-1}\cdots E_{1}E_{n}E_{n-1}\cdots E_{i+1}E_{i}
= \mu^{\{1,2,\cdots, n\}}I_{H_i}. 
$$

Take $x(\ne 0)\in H_{1}$. Since
$H_{1}\subset H_{j}$ for any $1\leq j\leq n$ and 
$E_{n}\cdots E_{1}=\mu^{\{1,2,\cdots, n\}}I_{H_{1}}$, we have that 
$x=\mu^{\{1,2,\cdots, n\}}x$.
Hence $\mu^{\{1,2,\cdots, n\}}=1.$

We shall show that
$$H_{1}=H_{2}=\cdots=H_{n}.$$
On the contrary we assume that $H_k \not= H_{\ell}$ for 
some $k \not= \ell$.

Using Lemma \ref{lemma:spacedecomposition},
we can represent $H_{i}$ as 
$$
H_{i}=K_{1}\oplus K_{2}\oplus \cdots \oplus K_{m(i)}
$$
and $m(1)=1$. 
Then there exists the smallest $i$ such that $m(i) > 1.$ 
We compare $m(j)$ and $m(i)$.
If there exists $m(j)$ such that 
$m(j)<m(i)(i\leq j\leq n)$.
Take $x(\ne 0)\in K_{m(i)}\subset H_{i}$.
Then
$$E_{j-1}E_{j-2}\cdots E_{i+1}E_{i}x=0.$$
This contradicts that
$$E_{i-1}\cdots E_{1}E_{n}E_{n-1}\cdots E_{i+1}E_{i}
= I_{H_i}.
$$ 

If there exists no  $m(j)$ such that 
$m(j)<m(i)(i\leq j\leq n)$.
Take $x(\ne 0)\in K_{m(i)}\subset H_{i}$.
Then
$E_{n-1}E_{n-2}\cdots E_{i+1}E_{i}x=x,$ and $E_{n}x=0.$ 
This also contradicts that
$$E_{i-1}\cdots E_{1}E_{n}E_{n-1}\cdots E_{i+1}E_{i}
= I_{H_i}.
$$ 
Therefore we have 
that
$$H_{1}=H_{2}=\cdots=H_{n}=:M.$$
Moreover we also have that 
$$E_{1}=E_{2}=\cdots=E_{n}=I_M.$$
In particular, 
$T_{i}^{\{1,2,\cdots, n\}}=I_M$ for any $i$ and 
$$
A_i = E_{i-1}\cdots E_{1}E_{n}E_{n-1}\cdots E_{i+1}A_i 
= T_{i}^{\{1,2,\cdots, n\}\setminus k} 
= \mu^{\{1,2,\cdots, n\}\setminus k}I_{H_k}.
$$

We shall show that  $\dim M= 1$. On the contrary, 
assume that $\dim M \geq 2$. Then there exists a non-scalar 
operator $B \in B(M)$. Since each $A_k$ is a scalar operator 
for any $k$, $(B, \dots, B)$ is in $End(H,f)$. This contradicts 
to that $(H,f)$ is transitive. Therefore $\dim M= 1$.  
Hence we may assume that 
$H_{i}=\mathbb{C}$ for any $i$. 
Since $(H,f)$ is transitive,
there  exists only one $(H,f)$-connected component on 
$V = \{1,2,\cdots, n\}$ by  Lemma \ref{lemma:component}.

Next we consider the case 
that there exists $H_{i}=0$ for some $i$.
If there exists only one vertex $i$ such that
$H_{i}\ne 0$, then 
$\dim H_{i}=1$ because $(H,f)$ is transitive.
Therefore we may assume that there exists 
 more than two vertices $i$ such that $H_{i}\ne 0$.

We consider the reduction of the quiver $C_{n}$
to the set of vertices $i$ with $H_{i}\ne 0$
to get another quiver $C_{m}(2\leq m\leq n)$.

Let $(K,g)$ be the reduced Hilbert representation 
of $C_{m}$ 
from the Hilbert representation $(H,f)$ of $C_{n}$ 
by Lemma \ref{lemma:End=}. 
Then 
$End(H,f)$ is isomorphic to $End(K,g)$.
Since $(H,f)$ is transitive, $(K,g)$ is also transitive.

Since we can adapt the above consideration to $(K,g)$, 
we have that $H_{i}=\mathbb{C}$ for all
$i$ such that $H_{i}\ne 0$.
Therefore in $(H,f)$,
we may and do have that $H_{i}=\mathbb{C}$ or 0 for $1\leq i\leq n$.
Since $(H,f)$ is transitive, 
by Lemma \ref{lemma:component},
there exists only one $(H,f)$-connected component 
$\{i \in V ; H_i \not= 0\}$.

\end{proof}


\begin{thm}
Let $\Gamma $ be a quiver whose underlying undirected graph is
an extended Dynkin diagram $\widetilde{A_{n}},(n\geq 0)$. 
Then there exists an infinite-dimensional transitive 
Hilbert representation of $\Gamma $ if and only if 
$\Gamma $ is not an oriented cyclic quiver.
\label{thm:An-transitive}
\end{thm}
\begin{proof}
Assume that $\Gamma $ is not an oriented cyclic quiver.
Then,
by Theorem \ref{thm:notcyclic}, there exists an infinite-dimensional transitive 
Hilbert representation of $\Gamma $.
Conversely assume that $\Gamma $ is an oriented cyclic quiver.
Then transitive Hilbert representations of $\Gamma $
are finite-dimensional
 by the above lemma \ref{lemma:cyclictransitive}.
Hence there exist no infinite-dimensional transitive 
Hilbert representations of $\Gamma $.

\end{proof}


Gabriel's theorem states that a  finite, connected quiver has 
only finitely many indecomposable 
representations if and only if 
the underlying undirected graph is one of Dynkin diagrams 
$A_n, D_n, E_6, E_7,E_8$. 
In \cite{EW3}, we 
constructed some examples of indecomposable, 
infinite-dimensional
representations of quivers  with the 
underlying  undirected graphs extended Dynkin diagrams 
$\tilde{D_n} \  (n \geq 4), \tilde{E_6},\tilde{E_7}$ and $\tilde{E_8}$. 
We used
the quivers whose
vertices are represented by a family of subspaces and 
whose arrows are  represented by natural inclusion maps.
Replacing the unilateral shift $S$ with
a transitive operator
in the construction of examples of indecomposable, 
infinite-dimensional
representations of quivers  in
\cite{EW3}, we shall give
 some examples of  
infinite-dimensional transitive
representations of quivers  with the 
underlying  undirected graphs extended Dynkin diagrams 
$\tilde{D_n} \  (n \geq 4), \tilde{E_6},\tilde{E_7}$ and $\tilde{E_8}$.
Our construction of examples is considered as a modification
of an unbounded operator used by 
Harrison,Radjavi and Rosenthal \cite{HRR}
to provide a transitive lattice.

\begin{lemma}
Let $\Gamma=(V,E,s,r)$ be the following quiver with the 
underlying  undirected graph an extended Dynkin diagram 
$\tilde{D_n}$ for $n \geq 4$ \rm{:}

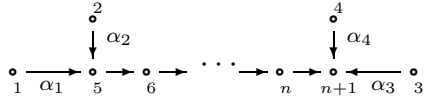
\begin{picture}(150,60)(-75,5)
\put(0,25){\thicklines\circle{2}}
\put(0,20){${}_{{}_{1}}$}

\put(5,25){\vector(1,0){20}}

\put(10,20){${}_{\alpha_{1}}$}

\put(30,25){\thicklines\circle{2}}

\put(30,20){${}_{{}_{5}}$}

\put(30,40){\vector(0,-1){10}}

\put(30,45){\thicklines\circle{2}}

\put(30,50){${}_{{}_{2}}$}

\put(35,38){${}_{\alpha_{2}}$}

\put(35,25){\vector(1,0){10}}

\put(50,25){\thicklines\circle{2}}

\put(50,20){${}_{{}_{6}}$}

\put(55,25){\vector(1,0){10}}

\put(70,25){$\cdots$}

\put(85,25){\vector(1,0){10}}

\put(100,25){\thicklines\circle{2}}

\put(100,20){${}_{{}_{n}}$}

\put(105,25){\vector(1,0){10}}

\put(120,25){\thicklines\circle{2}}

\put(115,20){${}_{{}_{n+1}}$}

\put(145,25){\vector(-1,0){20}}

\put(134,20){${}_{\alpha_{3}}$}

\put(150,25){\thicklines\circle{2}}

\put(150,20){${}_{{}_{3}}$}

\put(120,40){\vector(0,-1){10}}

\put(120,45){\thicklines\circle{2}}

\put(120,50){${}_{{}_{4}}$}

\put(125,38){${}_{\alpha_{4}}$}

\end{picture}

\noindent
Then there exists an infinite-dimensional, 
transitive Hilbert representation $(H,f)$ of $\Gamma$.  
\label{lemma:Dn}
\end{lemma}
\begin{proof}
Let $K = \ell^2(\mathbb N)$ and $S$ a transitive operator on $K$ with the domain $D(S)$. 
We define a Hilbert representation
$(H,f):=((H_v)_{v\in V}$
$,(f_{\alpha})_{\alpha \in E})$ 
of $\Gamma$ as follows: \\
Define 
 $H_1 = K \oplus 0$, 
$H_2 = 0 \oplus K$, 
$H_3 = \{(x,Sx)\in K \oplus K ; x \in D(S)\}$, \\
$H_4 = \{(x,x)\in K \oplus K ; x \in K\}$,  
$H_5 = H_6 = \dots =H_{n+1} = K \oplus K$. \\
Let $f_{\alpha_k}  
: H_{s(\alpha_k)}  \rightarrow  H_{r(\alpha_k)} $ be 
the inclusion map for any  $\alpha_k \in E$ for 
$k = 1,2,3,4$,  and $f_{\beta} = id$ for other arrows 
$\beta \in E$. 

Take $T = (T_v)_{v \in V} \in End(H,f)$. 
Since $T \in End (H,f)$ and any arrow is represented 
by the inclusion map, we have $T_i=T_j(i=5,\cdots,n+1),
T_5x = T_vx$ for any 
$v \in \{1,2,4\}$,
any  $x \in H_v$.
In particular, 
$T_5 H_v \subset H_v$($v \in \{1,2,4\})$,.  
Hence 
$T_{5}$ is written as
$T_{5}=A\oplus A$
as in \cite[Lemma 6.1,Example 3]{EW3}.
Moreover 
$H_3$ is also invariant under $T_5$. 
Since $S$ is transitive ,
 we have that
$A$ is a scalar by Lemma \ref{transitive operator}. 
Thus $T$ is a scalar, that is, 
$End(H,f)=\mathbb{C}.$
Therefore $(H,f)$ is transitive.

\end{proof}
\begin{lemma}
Let $\Gamma=(V,E,s,r)$ be the following quiver with the 
underlying  undirected graph an extended Dynkin diagram 
$\tilde{E_6}$ \rm{:}


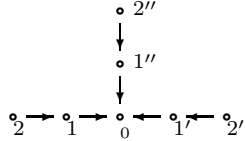
\begin{picture}(130,65)(-80,0)
\put(80,15){\thicklines\circle{2}}
\put(80,10){${}_{{}_{0}}$}

\put(80,30){\vector(0,-1){10}}
\put(80,35){\thicklines\circle{2}}
\put(85,35){${}_{1^{{\prime}{\prime}} }$}

\put(80,50){\vector(0,-1){10}}
\put(80,55){\thicklines\circle{2}}
\put(85,55){${}_{ 2^{{\prime}{\prime}}  }$}

\put(95,15){\vector(-1,0){10}}
\put(100,15){\thicklines\circle{2}}
\put(100,10){${}_{1^{{\prime}} }$}

\put(115,15){\vector(-1,0){10}}
\put(120,15){\thicklines\circle{2}}
\put(120,10){${}_{ 2^{{\prime}}}$}

\put(65,15){\vector(1,0){10}}
\put(60,15){\thicklines\circle{2}}
\put(60,10){${}_{1}$}

\put(45,15){\vector(1,0){10}}
\put(40,15){\thicklines\circle{2}}
\put(40,10){${}_{ 2}  $}

\end{picture}

\noindent
Then there exists an infinite-dimensional, 
transitive Hilbert representation $(H,f)$ of $\Gamma$.

\label{lemma:E6}
\end{lemma}
\begin{proof}
Let $(H,f)= ((H_v)_{v\in V},(f_{\alpha})_{\alpha \in E})$
be the following Hilbert representation  
of $\Gamma$: Let $K = \ell^2(\mathbb N)$ and 
$S$ a transitive operator on $K$ with the domain $D(S)$. 
Define $H_{0}=K \oplus K \oplus K,$
$H_{1}=0\oplus K \oplus K,$ \\
$H_{2}=0\oplus\{(y,Sy)\in K^2 ; \ y\in D(S)\},$
$H_{1^{^{\prime}}}=K \oplus K \oplus0,$ \\
$H_{2^{^{\prime}}}=\{(x,x)\in K^2 ; \ x\in K\}\oplus0,$
$H_{1^{^{\prime\prime}}}=K\oplus0\oplus K,$ \\
$H_{2^{^{\prime\prime}}}=\{(x,0,x) \in K^3 ; \ x\in K\}.$
For any arrow $\alpha \in E$, let  
$f_{\alpha} : H_{s(\alpha)} \rightarrow H_{r(\alpha)}$ be 
the canonical inclusion map. 
Take $T = (T_v)_{v \in V} \in End(H,f)$. 
Since any arrow is represented 
by the inclusion map, we have $T_0 x = T_vx$ for any 
$v \in \{1,1',2',1^{\prime\prime},2^{\prime\prime} \}$ and 
any  $x \in H_v$. In particular, 
$T_0 H_v \subset H_v$.  
Hence 
$T_{0}$ is written as
$T_{0}=A\oplus A\oplus A.$
Moreover 
$H_2
= \{(0,x,Sx) \in K^3 ; \ x \in D(S) \}$
is also invariant under $T_0$. 
Hence for any $x \in D(S)$, there exists $y \in D(S)$ such that 
$(0,Ax,ASx) = (0,y,Sy)$ as in \cite[Example 4]{EW3}. 
Since $S$ is transitive ,
 we have that
$A$ is a scalar by Lemma \ref{transitive operator}.  Thus $T$ is a scalar, that is, 
$End(H,f)=\mathbb{C}.$
Therefore $(H,f)$ is transitive.

\end{proof}
\begin{lemma}
Let $\Gamma=(V,E,s,r)$ be the following quiver with the 
underlying  undirected graph an extended Dynkin diagram 
$\tilde{E_7}$ \rm{:}

\begin{picture}(140,45)(-80,0)

\put(10,15){\thicklines\circle{2}}

\put(10,10){${}_{{}_{3}}$}

\put(15,15){\vector(1,0){10}}

\put(30,15){\thicklines\circle{2}}

\put(30,10){${}_{{}_{2}}$}

\put(35,15){\vector(1,0){10}}

\put(50,15){\thicklines\circle{2}}

\put(50,10){${}_{{}_{1}}$}

\put(55,15){\vector(1,0){10}}

\put(70,15){\thicklines\circle{2}}

\put(70,10){${}_{{}_{0}}$}

\put(85,15){\vector(-1,0){10}}

\put(90,15){\thicklines\circle{2}}

\put(90,10){${}_{{}_{1^{\prime}}   }$}

\put(105,15){\vector(-1,0){10}}

\put(110,15){\thicklines\circle{2}}

\put(110,10){${}_{{}_{2^{\prime} }  }$}

\put(125,15){\vector(-1,0){10}}

\put(130,15){\thicklines\circle{2}}

\put(130,10){${}_{{}_{3^{\prime} }  }$}

\put(70,30){\vector(0,-1){10}}

\put(70,35){\thicklines\circle{2}}

\put(75,35){${}_{{}_{ 1^{{\prime}{\prime}} }  }$}

\end{picture}

\noindent
Then there exists an infinite-dimensional, 
transitive Hilbert representation $(H,f)$ of $\Gamma$.  
\label{lemma:E7}
\end{lemma}
\begin{proof}
Let $K = \ell^2(\mathbb N)$ and $S$ 
a transitive operator on $K$ with the domain $D(S)$.
Define a Hilbert representation
$(H,f) := ((H_v)_{v\in V},(f_{\alpha})_{\alpha \in E})$ 
of $\Gamma$ as follows: \\
Let 
$H_{0}=K\oplus K\oplus K\oplus K,$
$H_{1}=K\oplus0\oplus K\oplus K,$ \\
$H_{2}=K\oplus0\oplus\{(x,x);x\in K\},$
$H_{3}=K\oplus0\oplus0\oplus0,$ \\
$H_{1^{^{\prime}}}=0\oplus K\oplus K\oplus K,$
$H_{2^{^{\prime}}}
=0\oplus K\oplus\{(y,Sy)\in K^2 ; y\in D(S)\},$ \\
$H_{3^{^{\prime}}}=0\oplus K\oplus0\oplus0$ and 
$H_{1^{^{\prime\prime}}}=\{(x,y,x,y)\in K^4 ; \ x,y \in K\}.$
For any arrow $\alpha \in E$, let  
$f_{\alpha} : H_{s(\alpha)} \rightarrow H_{r(\alpha)}$ be 
the canonical inclusion map.   
Take $T = (T_v)_{v \in V} \in End(H,f)$. 
Since any arrow is represented 
by the inclusion map, we have $T_0 x = T_vx$ for any 
$v \in \{1,2,3,1',2',3',1'' \}$ and 
any  $x \in H_v$. In particular, 
$T_0 H_v \subset H_v$.  
Hence 
$T_{0}$ is written as
$T_{0}=A\oplus A\oplus A\oplus A.$
Moreover 
$H_1 \cap H_{2^{^{\prime}}}
= \{(0,0,x,Sx) \in K^4 ; \ x \in D(S) \}$
is also invariant under $T_0$. 
Hence for any $x \in D(S)$, 
there exists $y \in D(S)$ such that 
$(0,0,Ax,ASx) = (0,0,y,Sy)$ as in \cite[Lemma 6.2]{EW3}. 
Since $S$ is transitive ,
 we have that
$A$ is a scalar by Lemma \ref{transitive operator}. Thus $T$ is a scalar, that is, 
$End(H,f)=\mathbb{C}.$
Therefore $(H,f)$ is transitive.
\end{proof}

\begin{lemma}

Let $\Gamma=(V,E,s,r)$ be the following quiver with the 
underlying  undirected graph an extended Dynkin diagram 
$\tilde{E_8}$ \rm{:}

\begin{picture}(150,45)(-70,0)

\put(10,15){\thicklines\circle{2}}

\put(10,10){${}_{{}_{5}}$}

\put(15,15){\vector(1,0){10}}

\put(30,15){\thicklines\circle{2}}

\put(30,10){${}_{{}_{4}}$}

\put(35,15){\vector(1,0){10}}

\put(50,15){\thicklines\circle{2}}

\put(50,10){${}_{{}_{3}}$}

\put(55,15){\vector(1,0){10}}

\put(70,15){\thicklines\circle{2}}

\put(70,10){${}_{{}_{2}}$}

\put(75,15){\vector(1,0){10}}

\put(90,15){\thicklines\circle{2}}

\put(90,10){${}_{{}_{1} }$}

\put(95,15){\vector(1,0){10}}

\put(110,15){\thicklines\circle{2}}

\put(110,10){${}_{{}_{0} }$}

\put(125,15){\vector(-1,0){10}}

\put(130,15){\thicklines\circle{2}}

\put(130,10){${}_{{}_{1^{\prime} }  }$}

\put(145,15){\vector(-1,0){10}}

\put(150,15){\thicklines\circle{2}}

\put(150,10){${}_{{}_{2^{\prime} }  }$}

\put(110,35){\thicklines\circle{2}}
\put(115,35){${}_{{}_{ 1^{{\prime}{\prime}} }  }$}
\put(110,30){\vector(0,-1){10}}

\end{picture}

\noindent
Then there exists an infinite-dimensional, transitive Hilbert representation $(H,f)$ of $\Gamma$.  

\label{lemma:E8}
\end{lemma}

\begin{proof}
Let $K = \ell^2(\mathbb N)$ and $S$ 
a transitive operator on $K$ with the domain $D(S)$. 
We define a Hilbert representation
$(H,f) := ((H_v)_{v\in V},(f_{\alpha})_{\alpha \in E})$ 
of $\Gamma$ as follows: \\
Let 
$H_{0}=K\oplus K\oplus K\oplus K\oplus K\oplus K,$ \\
$H_{1}=\{(x,x)\in K^2 ; \ x\in K\} 
\oplus K\oplus K\oplus K\oplus K,$ \\
$H_{2}=0\oplus0\oplus K\oplus K\oplus K\oplus K,$
$H_{3}=0\oplus0\oplus0\oplus K\oplus K\oplus K,$ \\
$H_{4}=0\oplus0\oplus0\oplus K\oplus\{(y,Sy)\in K^2 ; \ y\in D(S)\},$
$H_{5}=0\oplus0\oplus0\oplus K\oplus0\oplus0,$ \\
$H_{1^{^{\prime}}}=K\oplus K\oplus\{(x,y,x,y)\in K^4 ; \ x,y\in K\},$
$H_{2^{^{\prime}}}=K\oplus K\oplus0\oplus0\oplus0\oplus0,$ \\
$H_{1^{^{\prime\prime}}}=\{(y,z,x,0,y,z)\in K^6 ; \ x,y,z\in K\}.$\\
For any arrow $\alpha \in E$, let  
$f_{\alpha} : H_{s(\alpha)} \rightarrow H_{r(\alpha)}$ be 
the canonical inclusion map. 

Take $T = (T_v)_{v \in V} \in End(H,f)$. 
Since any arrow is represented 
by the inclusion map, we have $T_0 x = T_vx$ for any 
$v \in V $ and 
any  $x \in H_v$. In particular, 
$T_0 H_v \subset H_v$.  
Since $T_{0}$ preserves 
subspaces 
$H_{v},v=1,1',1",2,2',3,5$,
$T_{0}$ is written as  
\[
T_0 = A \oplus A \oplus A \oplus A \oplus A 
      \oplus A \oplus A \oplus A .
\]

Finally $\ H_{4}
=0\oplus0\oplus0\oplus K\oplus\{(y,Sy)\in K^2 ; \ y\in K\}$ 
is invariant under $T_0$. 
Then for any $x\in K $ and $y\in D(S)$,
 there exist $x'\in K$ and $y' \in D(S)$ such that 
\[
T_{0}(0,0,0,x,y,Sy)
=(0,0,0,Ax,Ay,ASy)=(0,0,0,x',y',Sy').
\]
Hence $ASy = Sy' = SAy$ as in \cite[Lemma 6.3]{EW3}. 
Since $S$ is transitive ,
 we have that
$A$ is a scalar by Lemma \ref{transitive operator}.
 
Thus $T = (T_v)_{v \in V}$ is a scalar, that is, 
$End(H,f)=\mathbb{C}.$
Therefore $(H,f)$ is transitive.

\end{proof}


Next,
we shall investigate the endomorphism algebras of 
Hilbert representations.
At first we recall some facts about reflection functors from \cite{EW3}.

Reflection functors are crucially used 
in the proof of the classification of finite-dimensional, 
indecomposable representations of tame quivers
(cf.\cite{As},\cite{BGP},\cite{DR},\cite{DF},\cite{GR},
\cite{GP}). As a matter of fact many 
indecomposable representations of tame quivers can be 
reconstructed by iterating  reflection functors on 
simple indecomposable representations. 
We can not expect such a best position in infinite-dimensional 
Hilbert representations. But reflection functors are still 
valuable to show that 
some property of representations of 
quivers on infinite-dimensional Hilbert spaces 
does not depend on the choice of orientations 
and 
does depend on the fact underlying undirected graphs 
are (extended) Dynkin diagrams or not. 

Let $\Gamma=(V,E,s,r)$ be a finite quiver. 
We say that a vertex $v \in V$ is a {\it sink} if $v \not= s(\alpha)$ 
for any $\alpha \in E$. Put 
$E^v = \{ \alpha \in E ; \ r(\alpha) = v \}$. 
We denote by $\overline{E}$ the set of all 
formally reversed new arrows $\overline{\alpha}$ for 
$\alpha \in E$. 
In this way if $\alpha : x \rightarrow y$ is an arrow, then 
$\overline{\alpha }: x \leftarrow y $. 

\smallskip
\noindent
{\bf Definition.}\cite{EW3}
Let $\Gamma=(V,E,s,r)$ be a finite quiver.  
For a sink $v \in V$, we construct  a new quiver 
$\sigma_v^+(\Gamma) = (\sigma_v^+(V), \sigma_v^+(E), s,r)$ 
as follows: All the arrows of $\Gamma$ having $v$ as range 
are reversed and all the other arrows remain unchanged.  
That is,\[
 \sigma_v^+(V) = V \ \ \ 
\sigma_v^+(E) = (E \setminus E^v) \cup \overline{E^v} , 
\]
where $\overline{E^v} = \{ \overline{\alpha } ;
\ \alpha \in E^v \}$.  

\smallskip
\noindent
{\bf Definition.} \cite{EW3} (reflection functor $\Phi_v^+$.) 
Let $\Gamma=(V,E,s,r)$ be a finite quiver.  
For a sink $v \in V$, we define a {\it reflection functor} at  $v$ 
\[
\Phi_v^+ :  HRep (\Gamma) \rightarrow  HRep(\sigma_v^+(\Gamma))
\]  
between the  categories  of Hilbert representations of $\Gamma$  
and $\sigma_v^+(\Gamma)$ as follows: 
For a Hilbert representation $(H,f)$ of $\Gamma$, 
we define a Hilbert representation 
$(K,g) = \Phi_v^+(H,f)$ of $\sigma_v^+(\Gamma)$.  
Let
\[
h_v: \oplus_{\alpha \in E^v} H_{s(\alpha)} \rightarrow H_v 
\]
be a bounded linear operator defined by 
\[
h_v((x_{\alpha})_{\alpha \in E^v}) 
= \sum_{\alpha \in E^v} f_{\alpha}(x_{\alpha}).
\]
We shall define 
\[
K_v := \Ker h_v = \{ (x_{\alpha})_{\alpha \in E^v} \in 
\oplus_{\alpha \in E^v} H_{s(\alpha)}
; \ 
      \sum_{\alpha \in E^v} f_{\alpha}(x_{\alpha}) = 0 \}. 
\]
We also consider the canonical inclusion map 
$i_v : K_v \rightarrow \oplus_{\alpha \in E^v} H_{s(\alpha)}$. 
For $u \in V$ with $u \not=v$, put $K_u = H_u$. 

For $\beta \in E^v$, let 
\[
P_{\beta} : \oplus_{\alpha \in E^v} H_{s(\alpha)} 
\rightarrow H_{s(\beta)}
\]
be the canonical projection. Then we shall define 
\[
g_{\overline{\beta}} :  K_{s(\overline{\beta})} = K_v \rightarrow 
K_{r(\overline{\beta})} = H_{s(\beta)} \ \ 
\text{ by } \  g_{\overline{\beta}} = P_{\beta} \circ i_v
\]
that is, 
$g_{\overline{\beta}}((x_{\alpha})_{\alpha \in E^v}) = x_{\beta}$. 

For $\beta \not\in E^v$, let $g_{\beta} = f_{\beta}$.  

For a homomorphism $T : (H,f) \rightarrow (H',f')$, 
we define a homomorphism 
\[
S = (S_u)_{u \in V}= \Phi_v^+(T) : (K,g) 
= \Phi_v^+(H,f) \rightarrow 
(K',g') = \Phi_v^+(H',f')
\] If  $u = v$, a bounded operator 
$S_v : K_v \rightarrow K_v'$ is given by 
\[
S_v((x_{\alpha})_{\alpha \in E^v}) 
= (T_{s(\alpha)}(x_{\alpha}))_{\alpha \in E^v}.
\]

It is easily seen that $S_v$ is well-defined and we have the 
following commutative diagram: 
\[
\begin{CD}
0 @>>> K_v @>i_v>> \oplus_{\alpha \in E^v} H_{s(\alpha)} 
 @>h_v>> H_v \\
@.  @V S_v VV  @V 
(T_{s(\alpha)})_{\alpha \in E^v} VV @V T_v VV \\
0 @>>> {K'}_v  @>{i'}_v>> \oplus_{\alpha \in E^v} {H'}_{s(\alpha)}  
@>h_v'>>{H'}_v
\end{CD}
\]

For other $u \in V$ with $u \not= v$,  put 
\[
S_u = T_u : K_u = H_u \rightarrow K_u' = H_u'. 
\] 

\smallskip

We also consider a dual of the above construction.  
We say that a vertex $v \in V$ is a {\it source } 
if $v \not= r(\alpha)$ 
for any $\alpha \in E$. Put 
$E_v = \{ \alpha \in E ; \ s(\alpha) = v \}$. 

\smallskip
\noindent
{\bf Definition.}\cite{EW3}
Let $\Gamma=(V,E,s,r)$ be a finite quiver.  
For a source $v \in V$, we shall construct  a new quiver 
$\sigma_v^-(\Gamma) = (\sigma_v^-(V), \sigma_v^-(E), s,r)$ 
as follows: All the arrows of $\Gamma$ having $v$ as source 
are reversed and all the other arrows remain unchanged.  
That is, 
\[
 \sigma_v^-(V) = V \ \ \ 
\sigma_v^-(E) = (E \setminus E_v) \cup \overline{E_v} , 
\]
where $\overline{E_v} = \{ \overline{\alpha } ; 
\ \alpha \in E_v \}$.

\smallskip
\noindent
{\bf Definition.}\cite{EW3} (reflection functor $\Phi_v^-$.)
Let $\Gamma=(V,E,s,r)$ be a finite quiver.  
For a source $v \in V$, we shall define a {\it reflection functor} at $v$ 
\[
\Phi_v^- :  HRep (\Gamma) \rightarrow  HRep(\sigma_v^-(\Gamma))
\]  
between the  categories  of Hilbert representations of $\Gamma$  
and $\sigma_v^-(\Gamma)$ as follows: 
For a Hilbert representation $(H,f)$ of $\Gamma$, 
we define a Hilbert representation 
$(K,g) = \Phi_v^-(H,f)$ of $\sigma_v^-(\Gamma)$.  
Let
\[
\hat{h}_v: H_v  \rightarrow \oplus_{\alpha \in E_v} H_{r(\alpha)}   
\]
be a bounded linear operator defined by 
\[
\hat{h}_v(x) 
= (f_{\alpha}(x))_{\alpha \in E_v} \ \text{ for } x \in H_v.
\]
We shall define
\[
K_v := (\Im \hat{h}_v)^{\perp} = \Ker \hat{h}_v^* \subset 
\oplus_{\alpha \in E_v} H_{r(\alpha)},  
\]
where $\hat{h}_v^* : \oplus_{\alpha \in E_v} H_{r(\alpha)} 
\rightarrow H_v$ is given 
$\hat{h}_v^*((x_{\alpha})_{\alpha \in E_v}) 
= \sum f_{\alpha}^*(x_{\alpha})$. 
For $u \in V$ with $u \not=v$, put $K_u = H_u$. 

Let $Q_v : \oplus_{\alpha \in E_v} H_{r(\alpha)} \rightarrow K_v$ 
be the canonical projection.  
For $\beta \in E_v$, let 
\[
j_{\beta} : H_{r(\beta)} 
\rightarrow \oplus_{\alpha \in E_v} H_{r(\alpha)}
\]
be the canonical inclusion. 
We shall define 
\[
g_{\overline{\beta}} :  K_{s(\overline{\beta})} = 
H_{r(\beta)} \rightarrow K_{r(\overline{\beta})} = K_v \ \ 
\text{ by }  g_{\overline{\beta}} = Q_v \circ j_{\beta} .
\]

For $\beta \not\in E_v$, let $g_{\beta} = f_{\beta}$.  

For a homomorphism $T : (H,f) \rightarrow (H',f')$, 
we shall define a homomorphism \[
S = (S_u)_{u \in V}= \Phi_v^-(T) : (K,g) 
= \Phi_v^-(H,f) \rightarrow 
(K',g') = \Phi_v^-(H',f').
\] 
For $u = v$, a bounded operator 
$S_v : K_v \rightarrow K_v'$ is given by 
\[
S_v((x_{\alpha})_{\alpha \in E_v}) 
= Q_v'((T_{r(\alpha)}(x_{\alpha}))_{\alpha \in E_v}),
\]
where $Q_v' : \oplus_{\alpha \in E_v} H'_{r(\alpha)} 
\rightarrow K'_v$ 
be the canonical projection. 

We have the 
following commutative diagram: 

\[
\begin{CD}
H_v @>\hat{h}_v>> \oplus_{\alpha \in E_v} H_{r(\alpha)} 
@>Q_v>> K_v @>>> 0 \\
@V T_v VV  @V \oplus_{\alpha \in E_v} T_{r(\alpha)} VV  @V S_v VV @.\\
{H'}_v @>{{\hat{h}}'}_v>> \oplus_{\alpha \in E_v} {H'}_{r(\alpha)} 
 @>{Q'}_v>> {K'}_v
@>>>0
\end{CD}
\]
For other $u \in V$ with $u \not= v$,  put 
\[
S_u = T_u : K_u = H_u \rightarrow K_u' = H_u'. 
\] 

We shall describe a relation between two (covariant) functors 
$\Phi_v^+$ and $\Phi_v^-$. We shall define another 
(contravariant) functor $\Phi^*$ at the beginning.

Let $\Gamma=(V,E,s,r)$ be a finite quiver.  We shall define 
the opposite quiver 
$\overline{\Gamma}=(\overline{V},\overline{E},s,r)$ 
by reversing all the arrows, more precisely, 
\[
 \overline{V} = V \ \ \text{ and  } \ \ 
\overline{E} = \{ \overline{\alpha} ; \ \alpha \in E \}.
\]
\smallskip
\noindent

{\bf Definition.}\cite{EW3} Let $\Gamma=(V,E,s,r)$ be a finite quiver 
and $\overline{\Gamma}=(\overline{V},\overline{E},s,r)$ 
its opposite quiver. 
We shall define a contravariant functor   
\[
\Phi^*  :  HRep (\Gamma) \rightarrow  HRep(\overline{\Gamma})
\]  
between the  categories  of Hilbert representations of $\Gamma$  
and $\overline{\Gamma}$ as follows: 
For a Hilbert representation $(H,f)$ of $\Gamma$, 
we define a Hilbert representation 
$(K,g) = \Phi^*(H,f)$ of $\overline{\Gamma}$ by 
\[
K_u = H_u  \text{ for } u \in V \ \ \ \text{ and } 
g_{\overline{\alpha}} = f_{\alpha}^* \text{ for } 
\alpha  \in E.
\]
For a homomorphism $T : (H,f) \rightarrow (H',f')$, 
we define a homomorphism 
\[
S = (S_u)_{u \in V} = \Phi^*(T) : (K',g') 
= \Phi^*(H',f') \rightarrow 
(K,g) = \Phi^*(H,f),
\]  by bounded operators 
$S_u : K_u'= H_u' \rightarrow K_u = H_u$ given by 
$S_u = T_u^*$. 

\begin{prop}\cite[Proposition 4.2.]{EW3}
Let $\Gamma=(V,E,s,r)$ be a finite quiver. 
If $v \in V$ is a source of $\Gamma$, then 
$v$ is a sink of $\overline{\Gamma}$, 
$\sigma_v^{-}(\Gamma)
=\overline{\sigma_v^{+}(\overline{\Gamma})}$ 
and  the following assertions hold: 
\begin{itemize}
\item [(1)] 
For a  Hilbert representation $(H,f)$ of $\Gamma$, 
\[
\Phi_v^{-}(H,f)=\Phi^* (\Phi_v^{+}(\Phi^*(H,f))).
\]
\item[(2)]
For a homomorphism $T : (H,f) \rightarrow (H',f')$, 
\[
\Phi_v^{-}(T)=\Phi^* (\Phi_v^{+}(\Phi^*(T))).
\]
\end{itemize}
\label{prop:+=*-*}
\end{prop}

\begin{prop}\cite[Proposition 4.3.]{EW3}
Let $\Gamma=(V,E,s,r)$ be a finite quiver. 
If $v \in V$ is a sink of $\Gamma$, then 
$v$ is a source of $\overline{\Gamma}$, 
$\sigma_v^{+}(\Gamma)
=\overline{\sigma_v^{-}(\overline{\Gamma})}$ 
and the following assertions hold:
\begin{itemize}
\item [(1)] 
For a  Hilbert representation $(H,f)$ of $\Gamma$, 
\[
\Phi_v^{+}(H,f)=\Phi^* (\Phi_v^{-}(\Phi^*(H,f))). 
\]
\item[(2)]
For a homomorphism $T : (H,f) \rightarrow (H',f')$, 
\[
\Phi_v^{+}(T)=\Phi^* (\Phi_v^{-}(\Phi^*(T))).
\]
\end{itemize}
\end{prop}

We shall investigate endomorphisms of 
Hilbert representations and 
its images of reflection functors.
In the case of infinite-dimensional Hilbert 
representations, 
we need assume a certain 
closedness condition at a sink or a source. 

\noindent  {\bf Definition.}\cite{EW3} Let $\Gamma=(V,E,s,r)$ be a finite quiver and 
 $v \in V$ a sink.  We recall that 
$E^v = \{\alpha ; \ r(\alpha) = v\}$. 
It is said that a Hilbert representation $(H,f)$  of $\Gamma$ 
is {\it closed} at $v$ if 
$\sum_{\alpha \in E^v} \Im f_{\alpha} \subset H_v$ 
is a closed subspace.  It is said that 
$(H,f)$ is {\it full} at $v$ if 
$\sum_{\alpha \in E^v} \Im f_{\alpha} = H_v$. 

\smallskip

\noindent {\bf Definition.}(\cite{EW3}) Let $\Gamma=(V,E,s,r)$ be a finite quiver and 
 $v \in V$ a source.  We recall that $E_v = \{\alpha | s(\alpha) = v\}$.  
It is said that a Hilbert representation $(H,f)$  of $\Gamma$ 
is {\it co-closed} at $v$ if 
$\sum_{\alpha \in E_v} \Im f_{\alpha}^* \subset H_v$ 
is a closed subspace.  It is said  that 
$(H,f)$ is {\it co-full} at $v$ if 
$\sum_{\alpha \in E_v} \Im f_{\alpha}^* = H_v$. 

We note that the properties of fullness,co-fullness,closedness and co-closedness are preserved under isomorphism of Hilbert representations.
\begin{lemma}
Let $\Gamma$ be a finite quiver 
and $v\in \Gamma$ a sink.
Let $(H,f)$ 
and $(K,g)$ be isomorphic Hilbert representations of $\Gamma$.
If $(H,f)$ is full (resp.closed ) at $v$, then 
$(K,g)$ is full (resp.closed ) at $v$.
\label{lemma iso-full}
\end{lemma}
\begin{proof}
Assume that  $(H,f)$ is full at $v$.
Since $(H,f)$ and $(K,g)$ are isomorphic,
there exists a family $S=(S_{u})_{u\in V}$ of bounded invertible operators such that  
$S_{r(\alpha)}f_{\alpha}
=g_{\alpha}S_{s(\alpha)}$ for $\alpha\in E$.
Take an element $y\in K_{v}$.
By the invertibility of $S_{v}$, there exists an element $x\in H_{v}$ such that
$S_{v}(x)=y$.
Since $(H,f)$ is full at $v$, there exist $x_{s(\alpha)}\in H_{s(\alpha)}$ such that
$\sum_{\alpha\in E^{v}}f_{\alpha}(x_{s(\alpha)})=x$.
We put $y_{s(\alpha)}:=S_{s(\alpha)}(x_{s(\alpha)})$.
Then

\begin{displaymath}
\sum_{\alpha\in E^{v}}  g_{\alpha}(y_{s(\alpha)})
=\sum_{\alpha\in E^{v}}  g_{\alpha}S_{s(\alpha)}(x_{s(\alpha)})
=\sum_{\alpha\in E^{v}}  S_{v}f_{\alpha}(x_{s(\alpha)})\\
\end{displaymath}
\begin{displaymath}
=S_{v}\sum_{\alpha\in E^{v}}  f_{\alpha}(x_{s(\alpha)})
=S_{v}(x)
=y
\end{displaymath}

Hence $(K,g)$ is full at $v$.

We can similarly prove that closedness property is preserved under
isomorphism of Hilbert representations.

\end{proof}
\begin{lemma}
Let $\Gamma$ be a finite quiver and
 $v\in V$  a source.
Let $(H,f)$ 
and $(K,g)$ be isomorphic Hilbert representations of $\Gamma$.
If $(H,f)$ is  co-full (resp.co-closed) at $v$, then 
$(K,g)$ is  co-full (resp.co-closed) at $v$.
\label{lemma iso-co-full}
\end{lemma}
\begin{proof}
Since $(H,f)$ and $(K,g)$ are isomorphic,
$\Phi^*(H,f)$ and $\Phi^*(K,g)$ are isomorphic.
Hence the case of co-fullness is reduced to  the case of fullness.
We can similarly prove that co-closedness property is preserved under
isomorphism of Hilbert representations.
\end{proof}

\noindent 

The following theorem is well known for finite-dimensional Hilbert spaces
(\cite[page289,5.7.Corollary]{As} and \cite[page16,Proposition 2.1]{DR}).
\begin{thm}
Let $\Gamma=(V,E,s,r)$ be a finite quiver and 
 $v \in V$ a sink. If  a Hilbert representation $(H,f)$ of 
$\Gamma$ is full at $v$, then 
the map $\Phi_{v}^{+} :End(H,f)\to End(\Phi_{v}^{+}(H,f))$
 is an isomorphism
as  $\mathbb{C}$-algebras.
\label{thm phi+}
\end{thm}
\begin{proof}
We put $(K,g):=\Phi_{v}^{+}(H,f)$.
The mapping $\Phi_{v}^{+}$ gives a mapping of 
$End(H,f)$ to $End(K,g)$.
At first we shall show that 
$\Phi_{v}^{+}$ is one to one.
Assume that $S:=\Phi_{v}^{+}(T)=0$ for $T\in End(H,f)$.
We have
$S_{u}=T_{u}=0(u\ne v)$.
From this we shall show that $T_{v}=0$.
Since $T\in End(H,f)$,
$T_{v}f_{\alpha}=f_{\alpha}T_{s(\alpha)}$ for 
$\alpha\in E^v=\{\alpha\in E;r(\alpha)=v\}$.
Hence, for 
$x_{\alpha}\in H_{s(\alpha)}$,
$$T_{v}(\sum_{\alpha\in E^v} f_{\alpha}(x_{\alpha}))=\sum_{\alpha\in E^v} f_{\alpha}T_{s(\alpha)}(x_{\alpha})=0.$$
Since $(H,f)$ is full at $v$,
$T_{v}=0$.
Thus $\Phi_{v}^{+}$ is one to one.
Next 
we shall show that $\Phi_{v}^{+}$ is onto.
Take $S=(S_{u})_{{u}\in V}\in End(K,g)$.
We put $T_{u}=S_{u}$ for $u\ne v$.
We shall define an operator $T_{v}:H_{v}\to H_{v}$
such that 
$T_{v}(\sum_{\alpha\in E^{v}}
f_{\alpha}(x_{\alpha}))
=\sum_{\alpha\in E^{v}}
f_{\alpha}(T_{s(\alpha)}
(x_{\alpha}))$ for 
$x_{\alpha}\in H_{s(\alpha)}$.
We need to show that
$T_{v}$ is well defined.
If there exists an element 
 $(x_{\alpha}^{\prime})_{\alpha\in E^{v}}\in \oplus _{\alpha\in E^{v}}H_{s(\alpha)}$
such that
$$\sum_{\alpha\in E^{v}}f_{\alpha}(x_{\alpha})
=
\sum_{\alpha\in E^{v}}
f_{\alpha}(x_{\alpha}^{\prime}),$$
then
we must show that
$$\sum_{\alpha\in E^{v}}f_{\alpha}T_{s(\alpha)}(x_{\alpha})
=
\sum_{\alpha\in E^{v}}
f_{\alpha}T_{s(\alpha)}(x_{\alpha}^{\prime}).$$
Since $\sum_{\alpha\in E^{v}}f_{\alpha}(x_{\alpha})
=
\sum_{\alpha\in E^{v}}
f_{\alpha}(x_{\alpha}^{\prime})$,
we have
$$h_{v}((x_{\alpha}-x_{\alpha}^{\prime})_{\alpha\in E^{v}})
=
\sum_{\alpha\in E^{v}}f_{\alpha}(x_{\alpha}-x_{\alpha}^{\prime})=0.$$
Hence
$(x_{\alpha}-x_{\alpha}^{\prime})_{\alpha\in E^{v}}\in \ker h_{v}
=K_{v}.$
Since $S_{v}:K_{v}\to K_{v}$,
we have 
$S_{v}((x_{\alpha}-x_{\alpha}^{\prime})_{\alpha\in E^{v}})\in \ker h_{v}
=K_{v}.$
Hence $h_{v}(
S_{v}((x_{\alpha}-x_{\alpha}^{\prime})_{\alpha\in E^{v}}))=0$.
Since $S\in End(K,g)$,
we have
$$S_{s(\alpha)}g_{\bar{\alpha}}=
g_{\bar{\alpha}}S_{v} \text{ for } \alpha\in E^{v}.$$
We have
$S_{s(\alpha)}g_{\bar{\alpha}}((x_{\beta}-x_{\beta}^{\prime})_{\beta\in E^{v}}))=
S_{s(\alpha)}(x_{\alpha}-x_{\alpha}^{\prime})
=
T_{s(\alpha)}(x_{\alpha}-x_{\alpha}^{\prime})
$
and
$g_{\bar{\alpha}}S_{v}
((x_{\beta}-x_{\beta}^{\prime})_{\beta\in E^{v}}))=
P_{\alpha}
S_{v}
((x_{\beta}-x_{\beta}^{\prime})_{\beta\in E^{v}})).$
Hence
$$T_{s(\alpha)}(x_{\alpha}-x_{\alpha}^{\prime})
=
P_{\alpha}
S_{v}
((x_{\beta}-x_{\beta}^{\prime})_{\beta\in E^{v}})).$$
Then
$$\sum_{\alpha\in E^{v}}  f_{\alpha}T_{s(\alpha)}
(x_{\alpha}-x_{\alpha}^{\prime})
=
\sum_{\alpha\in E^{v}}  f_{\alpha}
P_{\alpha}
S_{v}
((x_{\beta}-x_{\beta}^{\prime})_{\beta\in E^{v}})$$
and
$$\sum_{\alpha\in E^{v}}  f_{\alpha}
P_{\alpha}
S_{v}
((x_{\beta}-x_{\beta}^{\prime})_{\beta\in E^{v}} ))=
h_{v}(S_{v}((x_{\beta}-x_{\beta}^{\prime})_{\beta\in E^{v}} ))=0
.$$
This gives
$$\sum_{\alpha\in E^{v}} f_{\alpha}T_{s(\alpha)}
(x_{\alpha})=\sum_{\alpha\in E^{v}}  f_{\alpha}T_{s(\alpha)}(x_{\alpha}^{\prime}).$$
Thus $T_{v}$ is well defined.
Next we shall show that 
$$T_{v}f_{\alpha}(x)=f_{\alpha}T_{s(\alpha)}(x) \text{ for } x\in H_{s(\alpha)}.$$
Take and fix $x\in H_{s(\alpha)}$ for $\alpha\in E^{v}$.
For $\beta\in E^{v}$,
we put
$$
x_{\beta}= 
\begin{cases} x&  \ \ (\beta=\alpha),  \\
  0 &  \ \  (\beta\ne \alpha).
\end{cases}
$$
Since
$T_{v}(\sum_{\beta\in E^{v}}
f_{\beta}(x_{\beta}))
=\sum_{\beta\in E^{v}}
f_{\beta}(T_{s(\beta)}
(x_{\beta})),$
we have
$$T_{v}f_{\alpha}(x)
=\sum_{\beta}f_{\beta}T_{s(\beta)}(x_{\beta})
=
f_{\alpha}T_{s(\alpha)}(x).$$
Next we shall show that $T_{v}:H_{v}\to H_{v}$
is bounded.
We decompose $\oplus_{\alpha\in E^{v}}H_{s(\alpha)}
=
\ker h_{v}\oplus (\ker h_{v})^{\perp}
=
K_{v}\oplus K_{v}^{\perp}.$
By Banach invertibility theorem,
$h_{v}\vert_{(K_{v})^{\perp}}:(K_{v})^{\perp}\to H_v$
is a bounded invertible operator.
We shall show that there exists a positive constant $c$
such that 
$$\parallel T_{v}x\parallel\leqq c\parallel
x\parallel
\text{ for  any } x\in H_{v}.$$
Take $x=h((x_{\alpha})_{\alpha\in E^{v}})=
\sum_{x\in E^{v}}
f_{\alpha}(x_{\alpha}).$
We get
\begin{align*}
&\parallel T_{v}(x)\parallel
=\parallel \sum_{\alpha\in E^{v}}T_{v}(
f_{\alpha}(x_{\alpha}))\parallel 
=\parallel
\sum_{x\in E^{v}}
f_{\alpha}(T_{s(\alpha)}(x_{\alpha}))\parallel\\
&=\parallel((f_{\alpha}T_{s(\alpha)})_{\alpha\in E^{v}})
((x_{\alpha})_{\alpha\in E^{v}})
\parallel 
\leqq\parallel
((f_{\alpha}T_{s(\alpha)})_{\alpha\in E^{v}})\parallel
\parallel
((x_{\alpha})_{\alpha\in E^{v}})
\parallel \\
&=\parallel
((f_{\alpha}T_{s(\alpha)})_{\alpha\in E^{v}})
\parallel\parallel (h\vert_{K_{v}
^{\perp}})^{-1}\parallel
\parallel 
x
\parallel
\leqq
c\parallel
x\parallel
\\
\end{align*}
where 
$((f_{\alpha}T_{s(\alpha)})_{\alpha\in E^{v}})$
is a row matrix and
$$c:=\parallel
((f_{\alpha}T_{s(\alpha)})_{\alpha\in E^{v}})
\parallel 
\parallel 
(h\vert_{K_{v}
^{\perp}})^{-1}
\parallel.$$
Hence  $T_{v}$ is bounded.
Next we shall show that
$\Phi_{v}^{+}(T)=S.$
Since $S\in End(K,g)$,
$$S_{s(\alpha)}P_{\alpha}i_{v}=S_{s(\alpha)}g_{\bar{\alpha}}=
g_{\bar{\alpha}}S_{v}=P_{\alpha}i_{v}S_{v} \text{ for } \alpha\in E^{v}.$$
For $((x_{\alpha})_{\alpha\in E^{v}})\in K_{v}$,
we have
$$S_{v}((x_{\alpha}))
=(P_{\alpha}i_{v}S_{v}((x_{\alpha})))_{\alpha\in E^{v}}
=(S_{s(\alpha)}P_{\alpha}i_{v}((x_{\alpha})))_{\alpha\in E^{v}}
=(S_{s(\alpha)}(x_{\alpha})).$$
By the definition of 
$\Phi_{v}^{+}(T)$,
$(\Phi_{v}^{+}(T))_{u}=S_{u}=T_{u}$
for $u\ne v$.
For $u=v$ and
$((x_{\alpha})_{\alpha\in E^{v}})\in K_{v}$,
\begin{align*}
&(\Phi_{v}^{+}(T))_{v}((x_{\alpha})_{\alpha\in E^{v}})
=((T_{s(\alpha)}(x_{\alpha}))_{\alpha\in E^{v}})\\
&=((S_{s(\alpha)}(x_{\alpha}))_{\alpha\in E^{v}})
=S_{v}((x_{\alpha})_{\alpha\in E^{v}}).
\end{align*}
Thus 
$(\Phi_{v}^{+}(T))_{v}=S_{v}.$
Hence $\Phi_{v}^{+}(T)=S.$
Hence
 $\Phi_{v}^{+}$
is onto.
We conclude that
$End(H,f)\cong  End(\Phi_{v}^{+}(H,f))$
as  $\mathbb{C}$-algebras.

\end{proof}

\begin{cor}
Let $\Gamma=(V,E,s,r)$ be a finite quiver and 
 $v \in V$ a sink. 
Assume that  a Hilbert representation $(H,f)$ of 
$\Gamma$ is full at $v$.
If   
 $(H,f)$ is transitive (resp.indecomposable),then
$\Phi_{v}^{+}(H,f)$ is transitive(resp.indecomposable).

\end{cor}

The following theorem is well known for finite-dimensional Hilbert spaces
(\cite[page289,5.7.Corollary]{As} and \cite[page16,Proposition 2.1]{DR}).
\begin{thm}
Let $\Gamma=(V,E,s,r)$ be a finite quiver and 
 $v \in V$ a source. If  a Hilbert representation $(H,f)$ of 
$\Gamma$ is co-full at $v$, then  
$\Phi_{v}^{-}:End(H,f)\to End(\Phi_{v}^{-}(H,f))$
 is an isomorphism
as  $\mathbb{C}$-algebras.
\label{thm phi-}
\end{thm}

\begin{proof}
We put $(K,g):=\Phi_{v}^{-}(H,f)$.
The mapping $\Phi_{v}^{-}$ gives a mapping of 
$End(H,f)$ to $End(K,g)$.
At first we shall show that 
$\Phi_{v}^{-}$ is one to one.
Assume that $S:=\Phi_{v}^{-}(T)=0$ for $T\in End(H,f)$.
We shall show that $T_{v}=0$.
Since $T\in End(H,f)$,
$$f_{\alpha}T_{v}=T_{r(\alpha)}f_{\alpha}\text{ for }\alpha\in E_{v}.$$
For $(x_{\alpha})_{\alpha\in E_{v}}\in \oplus_{\alpha\in E_{v}} H_{r(\alpha)}$, we have
$$T_{v}^{*}(\sum f_{\alpha}^{*}(x_{\alpha}))=\sum f_{\alpha}^{*}
(T_{r(\alpha)}^{*}
(x_{\alpha}))=
\sum f_{\alpha}^{*}
(S_{r(\alpha)}^{*}
(x_{\alpha}))=0 .$$
Since $(H,f)$ is co-full at $v$,
$T_{v}^{*}=0$.
Hence
$T_{v}=0$.
Thus $\Phi_{v}^{-}$ is one to one.
Next 
we shall show that $\Phi_{v}^{-}$ is onto.
We put $T_{u}=S_{u}$ for $u\ne v$.
And we shall define an operator $W_{v}:H_{v}\to H_{v}$
such that for $(x_{\alpha})_{\alpha\in E_{v}}\in \oplus _{\alpha\in E_{v}}H_{r(\alpha)}$,
$$W_{v}(\sum_{\alpha\in E_{v}}
f_{\alpha}^{*}(x_{\alpha}))
=\sum_{\alpha\in E_{v}}
f_{\alpha}^{*}(T_{r(\alpha)}^{*}
(x_{\alpha})).$$
We need to show that
$W_{v}$ is well defined.
Assume that  there exists an element 
 $(x_{\alpha}^{\prime})_{\alpha\in E_{v}}\in \oplus _{\alpha\in E_{v}}H_{r(\alpha)}$
such that
$$\sum_{\alpha\in E_{v}}f_{\alpha}^{*}(x_{\alpha})
=
\sum_{\alpha\in E_{v}}
f_{\alpha}^{*}(x_{\alpha}^{\prime}).$$
We have
$$\hat{h_{v}}^*((x_{\alpha}-x_{\alpha}^{\prime}))
=
\sum_{\alpha\in E_{v}}f_{\alpha}^{*}(x_{\alpha}-x_{\alpha}^{\prime})=0.$$
Hence
$(x_{\alpha}-x_{\alpha}^{\prime})_{\alpha\in E_{v}}\in \ker \hat{h_{v}}^*
=K_{v}.$
Since $S_{v}^{*}:K_{v}\to K_{v}$,
we have 
$S_{v}^{*}((x_{\alpha}-x_{\alpha}^{\prime})_{\alpha\in E_{v}})\in K_{v}
.$
Hence $\hat{h_{v}}^*(
S_{v}^{*}((x_{\alpha}-x_{\alpha}^{\prime})_{\alpha\in E_{v}}))=0$.
Since $S\in End(K,g)$,
we have
$$S_{v}g_{\bar{\beta}}=g_{\bar{\beta}}S_{r(\beta)}
\text{ and }
g_{\bar{\beta}}^{*}S_{v}^{*}=
S_{r(\beta)}^{*}g_{\bar{\beta}}^{*}.$$
Hence
$$
g_{\bar{\beta}}^{*}S_{v}^{*}
((x_{\alpha}-x_{\alpha}^{\prime})_{\alpha\in E_{v}})
=
P_{r(\beta)}i_{v}S_{v}^{*}
((x_{\alpha}-x_{\alpha}^{\prime})_{\alpha\in E_{v}})
,$$

$$S_{r(\beta)}^{*}g_{\bar{\beta}}^{*}
((x_{\alpha}-x_{\alpha}^{\prime})_{\alpha\in E_{v}})
=
S_{r(\beta)}^{*}
(x_{\beta}-x_{\beta}^{\prime})
.$$
Thus we have 
$$P_{r(\beta)}i_{v}S_{v}^{*}
((x_{\alpha}-x_{\alpha}^{\prime})_{\alpha\in E_{v}})
=
S_{r(\beta)}^{*}
(x_{\beta}-x_{\beta}^{\prime})
$$
and
$$\sum f_{\beta}^{*}
P_{r(\beta)}i_{v}
(S_{v}^{*}
((x_{\alpha}-x_{\alpha}^{\prime})_{\alpha\in E_{v}}))
=
\sum f_{\beta}^{*}S_{r(\beta)}^{*}
(x_{\beta}-x_{\beta}^{\prime})
.$$
Since
$$\sum f_{\beta}^{*}
P_{r(\beta)}i_{v}
(S_{v}^{*}
((x_{\alpha}-x_{\alpha}^{\prime})_{\alpha\in E_{v}}))
=\hat{h_{v}}^*(S_{v}^{*}
((x_{\beta}-x_{\beta}^{\prime})_{\beta\in E_{v}}))=0,$$
$$\sum f_{\beta}^{*}S_{r(\beta)}^{*}
(x_{\beta}-x_{\beta}^{\prime})
=
\sum_{\beta\in E_{v}} f_{\beta}^{*}T_{r(\beta)}^{*}
(x_{\beta}-x_{\beta}^{\prime})
=0.$$
Hence
$$\sum f_{\beta}^{*}T_{r(\beta)}^{*}
(x_{\beta})=
\sum f_{\beta}^{*}T_{r(\beta)}^{*}
(x_{\beta}^{\prime}).$$
Thus $W_{v}$ is well defined.
Put $T_{v}=W_{v}^{*}$.
Next we shall show that 
$$f_{\alpha}T_{v}=T_{s(\alpha)}f_{\alpha}
\text{ and }
T_{v}^{*}f_{\alpha}^{*}=f_{\alpha}^{*}T_{s(\alpha)}^{*}.$$
Take and fix $x\in H_{r(\alpha)}$.
For $\beta\in E_{v}$,
we put
$$
x_{\beta}= 
\begin{cases} x&  \ \ (\beta=\alpha),
 \\
  0 &  \ \  (\beta\ne \alpha).
\end{cases}
$$
By the definition of $W_{v}=T_{v}^{*}$,
$$W_{v}(\sum_{\alpha\in E^{v}}
f_{\alpha}^{*}(x_{\alpha}))
=\sum_{\alpha\in E^{v}}
f_{\alpha}^{*}(T_{r(\alpha)}^{*}
(x_{\alpha})).$$
Hence 
$$T_{v}^{*}f_{\alpha}^{*}(x)
=\sum_{\beta}f_{\beta}^{*}T_{r(\beta)}^{*}(x_{\beta})
=
f_{\alpha}^{*}T_{r(\alpha)}^{*}(x)\text{ for }x\in H_{r(\alpha)}.$$
Thus we proved it.
Next we shall show that $W_{v}=T_{v}^{*}:H_{v}\to H_{v}$
is bounded.
By Banach invertibility theorem,
$\hat{h}_{v}^*\vert_{(K_{v})^{\perp}}:
(K_{v})^{\perp}\to H_v$
is a bounded invertible operator.
We shall show that there exists a positive constant $c$
such that 
$$\parallel T_{v}^{*}x\parallel\leqq c\parallel
x\parallel
\text{ for  any } x\in H_{v}.$$
For $x\in H_{v}$,
there exists $(x_{\alpha})_{\alpha\in E_{v}}\in (K_{v})^{\perp}$
such that
$x=\hat{h}_{v}^*((x_{\alpha})_{\alpha\in E_{v}})=
\sum_{\alpha\in E_{v}}
f_{\alpha}^{*}(x_{\alpha})
.$
We have
\begin{align*}
&\parallel T_{v}^{*}(x)\parallel
=\parallel \sum_{\alpha\in E_{v}}T_{v}^{*}(
f_{\alpha}^{*}(x_{\alpha}))\parallel 
=\parallel
\sum_{\alpha\in E_{v}}
f_{\alpha}^{*}(T_{r(\alpha)}^{*}(x_{\alpha}))\parallel\\
\\
&=\parallel(f_{\alpha}^{*}T_{r(\alpha)}^{*})_{\alpha\in E_{v}}
(x_{\alpha})_{\alpha\in E_{v}}
\parallel 
\leqq\parallel
(f_{\alpha}^{*}T_{r(\alpha)}^{*})_{\alpha\in E_{v}}\parallel
\parallel
(x_{\alpha})_{\alpha\in E_{v}}
\parallel \\
&=\parallel
(f_{\alpha}^{*}T_{r(\alpha)}^{*})_{\alpha\in E_{v}}
\parallel\parallel (\hat{h}_{v}^*\vert_{K_{v}
^{\perp}})^{-1}\parallel
\parallel 
x
\parallel
\leqq
c\parallel
x\parallel
\\
\end{align*}
where 
$(f_{\alpha}^{*}T_{r(\alpha)}^{*})_{\alpha\in E_{v}}$
is a row matrix
and
$c:=\parallel
((f_{\alpha}^{*}T_{r(\alpha)}^{*}))_{\alpha\in E_{v}}
\parallel\parallel 
(\hat{h}_{v}^*\vert_{K_{v}
^{\perp}})^{-1}
\parallel.$
Hence  $T_{v}$ is bounded.
Next we shall show that
$\Phi_{v}^{-}(T)=S.$
By the definition of 
$\Phi_{v}^{-}(T)$,
$(\Phi_{v}^{-}(T))_{u}=S_{u}=T_{u}$
for $u(\ne v)\in V$.
Since $S\in End(K,g)$,
we have
$$S_{v}Q_{v}j_{\beta}=S_{v}g_{\bar{\beta}}
=g_{\bar{\beta}}S_{r(\beta)}=Q_{v}j_{\beta}S_{r(\beta)}
\text{ for }\beta\in E_{v}
.$$
For $(x_{\beta})_{\beta\in E_{v}}\in K_{v}$,we have
\begin{align*}
&S_{v}((x_{\beta})_{\beta\in E_{v}})
=S_{v}Q_{v}(\sum_{\beta\in E_{v}} j_{\beta} (x_{\beta}))
=\sum_{\beta\in E_{v}}S_{v}Q_{v} j_{\beta} (x_{\beta})
\\
&=\sum_{\beta\in E_{v}}
Q_{v}j_{\beta}
(S_{r(\beta)}x_{\beta})
=
Q_{v}
\sum_{\beta\in E_{v}}j_{\beta}(S_{r(\beta)}x_{\beta})
=Q_{v}((S_{r(\beta)}x_{\beta})_{\beta\in E_{v}}).
\end{align*}
Thus 
$$
S_{v}((x_{\beta})_{\beta\in E_{v}})
=Q_{v}((S_{r(\beta)}x_{\beta})_{\beta\in E_{v}}).
$$
For $u=v$ and
$((x_{\alpha})_{\alpha\in E_{v}})\in K_{v}$,
\begin{align*}
(\Phi_{v}^{-}(T))_{v}((x_{\alpha})_{\alpha\in E_{v}})
&=Q_{v}((T_{r(\alpha)}x_{\alpha})_{\alpha\in E_{v}}))
=Q_{v}((S_{r(\alpha)}x_{\alpha})_{\alpha\in E_{v}}))\\
&=S_{v}((x_{\alpha})_{\alpha\in E_{v}}).
\end{align*}
Thus 
$(\Phi_{v}^{-}(T))_{v}=S_{v}.$
Hence $\Phi_{v}^{-}(T)=S.$
Hence
 $\Phi_{v}^{-}:End(H,f)\to  End(\Phi_{v}^{-}(H,f))$
is onto.
Thus we have 
$End(H,f)\cong  End(\Phi_{v}^{-}(H,f))$
as $\mathbb{C}$-algebras.
\end{proof}
\begin{cor}
Let $\Gamma=(V,E,s,r)$ be a finite quiver and 
 $v \in V$ a source. Assume that  a Hilbert representation $(H,f)$ of 
$\Gamma$ is co-full at $v$.
If   
 $(H,f)$ is transitive (resp.indecomposable),
then $\Phi_{v}^{-}(H,f)$ is transitive (resp.indecomposable).
\end{cor}
Next, we shall show the existence of infinite-dimensional transitive Hilbert
representations of quivers with any orientation  whose  
underlying  undirected graphs  are extended Dynkin diagrams 
$\tilde{D_n} \  (n \geq 4), \tilde{E_6},\tilde{E_7}$ and $\tilde{E_8}$.

We recall some definitions and lemmas 
in \cite{EW3}.

\noindent
{\bf Definition.}\cite{EW3}
Let $\Gamma$ be a quiver whose underlying undirected graph 
is Dynkin diagram $A_n$. We count the arrows from the left as 
$\alpha_k : s(\alpha_k) \rightarrow r(\alpha_k), 
\ (k = 1, \dots, n-1)$.    
Let $(H,f)$ be a Hilbert representation of $\Gamma$. We denote 
$f_{\alpha_k}$ briefly by $f_k$. For example, 
\[
\circ_{H_1} \overset{f_1}{\longrightarrow}
 \circ_{H_2} \overset{f_2}\longrightarrow \circ_{H_3} 
\overset{f_3}{\longleftarrow} 
\circ_{H_4} \overset{f_4}\longrightarrow \circ_{H_5} 
\overset{f_5}\longleftarrow \circ_{H_6}
\]
It is said that $(H,f)$ is {\it positive-unitary diagonal} 
if there exist $m \in \mathbb N$ and 
orthogonal decompositions (admitting zero components) 
of Hilbert spaces 
\[
 H_k = \oplus _{i = 1}^m H_{k,i} \ \ \ (k = 1, \dots, n)
\]
and decompositions of operators 
\[
f_k =  \oplus _{i = 1}^m f_{k,i}:  
\oplus _{i = 1}^m H_{s(\alpha_k),i} 
\rightarrow \oplus _{i = 1}^m H_{r(\alpha_k),i} \  
\ \ \ (k = 1, \dots, n)
\]
such that each 
$f_{k,i} : H_{s(\alpha_k),i} \rightarrow H_{r(\alpha_k),i}$ 
is written as  $f_{k,i} = 0$ or 
$f_{k,i} = \lambda _{k,i} u_{k,i}$ for 
some positive scalar $\lambda _{k,i}$ and onto unitary 
$u_{k,i} \in B(H_{s(\alpha_k),i},H_{r(\alpha_k),i})$. 

It is easily seen that if $(H,f)$ is positive-unitary diagonal, 
then $\Phi^*(H,f)$ is also positive-unitary diagonal. 

\begin{lemma}\cite[Lemma 6.4.]{EW3}
Let $\Gamma$ be a quiver whose underlying undirected graph 
is Dynkin diagram $A_n$ and 
$(H,f)$ be a Hilbert representation of $\Gamma$. 
Suppose that $(H,f)$ is positive-unitary diagonal. Then 
$(H,f)$ is closed at any sink of $\Gamma$ and 
co-closed at any source of $\Gamma$. 
\label{lemma:positive-unitary-coslosed} 

\end{lemma}

\begin{prop}\cite[ Proposition 6.5.]{EW3}
Let $\Gamma$ be a quiver whose underlying undirected graph 
is Dynkin diagram $A_n$ and 
$(H,f)$ be a Hilbert representation of $\Gamma$. 
Let $v$ be a source of $\Gamma$. 
Suppose that $(H,f)$ is positive-unitary diagonal. 
Then $\Phi_v^{-}(H,f)$ is also positive-unitary diagonal. 
\label{prop:positive-unitary-iteration} 

\end{prop}

It is known that 
 every orientation of Dynkin 
diagram $A_n$ is obtained by  an iteration of 
$\sigma_v^-$ at sources $v$ except the right end
from a particular orientation as follows:

\begin{lemma}\cite[Lemma 6.6.]{EW3}
Let $\Gamma _0$  and $\Gamma$ be quivers
 whose underlying undirected graphs
are the same Dynkin diagram $A_n$ for $n \geq 2$. Assume that 
$\Gamma _0$ is the following: 
\[
\circ_1 \longrightarrow \circ_2 \longrightarrow \circ_3 \dots 
\circ_{n-1} \longrightarrow \circ_n
 \]
Then there exists a sequence $v_1, \dots, v_m$ of vertices in 
$\Gamma _0$ such that 
\begin{itemize}
\item [(1)] for each $k= 1, \dots, m$, $v_k$ is a source in 
$\sigma^-_{v_{k-1}} \dots \sigma^-_{v_2} \sigma^-_{v_1}(\Gamma _0)$,
\item [(2)] 
$\sigma^-_{v_m} \dots \sigma^-_{v_2} \sigma^-_{v_1}(\Gamma _0) = 
\Gamma$, 
\item[(3)]for each $k= 1, \dots, m$, $v_k \not= n$. 
\end{itemize}
\label{lemma:orientation-change} 
\end{lemma}
\begin{lemma}\cite[Lemma 5.6.]{EW3}
Let $\Gamma=(V,E,s,r)$ be a finite quiver and 
 $v \in V$ a sink. Then for any  Hilbert representation $(H,f)$ of 
$\Gamma$,  $\Phi_v^{+}(H,f)$ is co-full at $v$. 
\label{lemma:co-full}
\end{lemma}
\begin{thm}\cite[Theorem 5.13.]{EW3}
Let $\Gamma=(V,E,s,r)$ be a finite quiver and 
 $v \in V$ a source. Assume that 
 a Hilbert representation $(H,f)$ of $\Gamma$ is indecomposable 
and  co-closed at $v$. Then the following assertions hold:
\begin{itemize}
\item [(1)] 
If $\Phi_v^{-}(H,f) = 0$, then $H_v = \mathbb C$, 
$H_u = 0$ for any $u \in V$ with $u \not= v$ and 
$f_{\alpha} = 0$ for any $\alpha \in E$. 
\item[(2)]
If $\Phi_v^{-}(H,f) \not= 0$, then $\Phi_v^{-}(H,f)$ 
is also indecomposable  and 
$(H,f) \cong \Phi_v^{+}\Phi_v^{-}(H,f)). $
\end{itemize}
\label{thm:preserving-indecomposability} 

\end{thm}


The following is one of the main theorem in this paper.

\begin{thm}
Let $\Gamma$ be a quiver whose underlying undirected graph is an extended Dynkin diagram.
Then there exists an infinite-dimensional transitive Hilbert representation of $\Gamma$ if and only if $\Gamma$ is not an oriented cyclic quiver. 
\end{thm}
\begin{proof}
Suppose that $\Gamma$ is an oriented cyclic quiver.
Theorem \ref{thm:An-transitive} proves the nonexistence of infinite-dimensional transitive Hilbert representation of $\Gamma$.
Suppose that $\Gamma$ is not an oriented cyclic quiver.
We shall prove the existence of infinite-dimensional 
transitive Hilbert representations of $\Gamma$.
When $\widetilde{A_{n}}$ case, theorem
\ref{thm:An-transitive} proves the existence 
of infinite-dimensional 
transitive Hilbert representations of $\Gamma$.
Next we consider the case that the $|\Gamma|$ is $\tilde{D_n}$.  
Let $\Gamma _0$ 
be the quiver of Lemma \ref{lemma:Dn}  and $(H^{(0)},f^{(0)})$ 
the Hilbert representation constructed there. Then  
$|\Gamma _0| =  |\Gamma| = \tilde{D_n}$, but their orientations 
are different in general. Let $\Gamma _1$ be a  
quiver such that $|\Gamma _1| = \tilde{D_n}$ and the orientation 
is as same as $\Gamma $ on the path between 5 and n+1 and 
as same as $\Gamma_0$ on the rest four "wings". We shall define a Hilbert 
representation $(H^{(1)},f^{(1)})$ of $\Gamma _1$ 
modifying $(H^{(0)},f^{(0)})$. 
We put
$f^{(1)}_{\beta} = I$ for any arrow $\beta$ in the path between 5 and n+1. 
and $f^{(0)}_{\beta} = f^{(1)}_{\beta} $ for other arrow $\beta$.
The same proof for 
$(H^{(0)},f^{(0)})$ shows that $(H^{(1)},f^{(1)})$ is transitive. 
Since $f^{(1)}_{\alpha_{i}}(i=1,\cdots,4)$ is an inclusion map,
  $(H^{(1)},f^{(1)})$ is co-full at sources 
1,2,3 and 4.
By  the theorem \ref{thm phi-}, 
 a certain iteration of reflection functors at a source 
1,2,3 or 4 on $(H^{(1)},f^{(1)})$ gives an 
infinite-dimensional, transitive, 
 Hilbert representation of $\Gamma$.  
We have proved this case. 

Next we consider the case that the $|\Gamma|$ is $\tilde{E_6}$. 
Let $\Gamma _0$ 
be the quiver of Lemma \ref{lemma:E6}, and  we denote here by
$(H^{(0)},f^{(0)})$ 
the Hilbert representation constructed there. Then  
$|\Gamma _0| =  |\Gamma| = \tilde{E_6}$, but their orientations 
are different in general. 
Three "wings" of $|\Gamma _0|$ 
$2-1-0,\  2'-1'-0, \ 2''-1''-0$ 
can be regarded as  Dynkin diagrams $A_3$. 
Applying Lemma  \ref{lemma:orientation-change} for 
these wings locally, we can 
find a sequence $v_1, \dots, v_m$ of vertices in 
$\Gamma _0$ such that 
\begin{itemize}
\item [(1)] for each $k= 1, \dots, m$, $v_k$ is a source in 
$\sigma^-_{v_{k-1}} \dots \sigma^-_{v_2} \sigma^-_{v_1}(\Gamma _0)$, 
\item [(2)] 
$\sigma^-_{v_m} \dots \sigma^-_{v_2} \sigma^-_{v_1}(\Gamma _0) = \Gamma$, 
\item[(3)]for each $k= 1, \dots, m$, $v_k \not= 0$. 
\end{itemize}
We note that co-closedness of Hilbert representations at 
a source can be checked locally around the source. 
Since the restriction of the representation $(H^{(0)},f^{(0)})$ 
to each "wing" is positive-unitary diagonal and the iteration of 
reflection functors does not move the vertex $0$, we can apply 
Lemma \ref{lemma:positive-unitary-coslosed} and 
Proposition \ref{prop:positive-unitary-iteration}
 locally 
that $\Phi^-_{v_{k-1}} \dots $
$
\Phi^-_{v_2} \Phi^-_{v_1}(H^{(0)},f^{(0)})$ is 
co-closed at $v_k$ for $k= 1, \dots, m$. 
Since the particular Hilbert space $H_0^{(0)}$ associated with 
the vertex $0$ is infinite-dimensional and remains unchanged 
under the iteration of the reflection functors above, 
$\Phi^{-}_{v_{i}}\cdots \Phi^{-}_{v_{1}}(H^{(0)},f^{(0)})(1\leq i\leq m)$ is 
infinite-dimensional.
Therefore
the theorem \ref{thm:preserving-indecomposability} implies that  
$$\Phi^{-}_{v_{i}}\cdots \Phi^{-}_{v_{1}} (H^{(0)},f^{(0)}) (1\leq i\leq m)$$
is infinite-dimensional indecomposable Hilbert
representation of $\sigma^-_{v_i}$ $\dots \sigma^-_{v_2} \sigma^-_{v_1}(\Gamma)$ .
By Theorem \ref{thm:preserving-indecomposability} ,
we have ,
for $$(K,g):=\Phi^{-}_{v_{i}} \cdots \Phi^{-}_{v_{1}}(H^{(0)},f^{(0)})
 (1\leq i\leq m),$$
$$(K,g)\cong 
\Phi^{+}_{v_{i+1}}\Phi^{-}_{v_{i+1}}(K,g).$$
On the other hand,
by Lemma \ref{lemma:co-full},
$\Phi^{+}_{v_{i+1}}\Phi^{-}_{v_{i+1}}(K,g)$
is co-full at $v_{i+1}$.
Since $(K,g)\cong 
\Phi^{+}_{v_{i+1}}\Phi^{-}_{v_{i+1}}(K,g)$,
by  Lemma \ref{lemma iso-co-full},
we have that
$(K,g)$ is co-full at $v_{i+1}$.
Hence  Theorem \ref{thm phi-}
implies that
$End(K,g)\cong 
End(\Phi^{-}_{v_{i+1}}(K,g))$.
By induction,
we have
$$End(H^{(0)},f^{(0)})\cong End(\Phi^{-}_{v_{m}}\cdots 
\Phi^{-}_{v_{1}}(H^{(0)},f^{(0)})).$$
Since $(H^{(0)},f^{(0)})$ is transitive,
$(\Phi^{-}_{v_{m}}\cdots 
\Phi^{-}_{v_{1}}(H^{(0)},f^{(0)}))$ is also transitive.
Thus 
there exist infinite-dimensional transitive Hilbert representations for quivers with any orientation whose underlying undirected graphs is 
 extended Dynkin diagram $\widetilde{E_{6}}$.
The other cases $\tilde{E_7}$
and $\tilde{E_8}$ 
 are proved similarly.
\end{proof}


\begin{thebibliography}{1}
\bibitem[Ak]{Ak}
N.I.Akhiezer and I.M.Glazman,
Theory of Linear Operators in Hilbert Space.
Ungar,New York(1961).
\bibitem[Ar]{Ar}
N.Aronszajn and U.Fixman , {\it Algebraic spectral problems,}
Studia Math. 30 (1968), 273-338.
\bibitem[As]{As}
I.Assem,D.Simson and A.Skowronski,
Elements of the Representation theory of associative algebras
volume 1 Techniques of Representation Theory, 
Cambridge university press.

\bibitem[Au]{Au} M. Auslander, {\it Large modules over artin algebras,}
In: Algebra, Topology and Category Theory, Academic Press New York (1976),
1-17. 

\bibitem[BGP]{BGP} I. N. Bernstein, I. M. Gelfand and V. A. Ponomarev, 
Coxeter functors and Gabriel's theorem, Russian Math. Surveys, 
28 (1973), 17-32.  


\bibitem[DR]{DR} V. Dlab and C. M. Ringel, {\it Indecomposable 
representations of graphs and algebras,} Memoirs Amer. Math. 
Soc. 6 (1976), no. 173. 

\bibitem[DF]{DF} P. Donovan and M. R. Freislish, 
{\it The representation theory of finite graphs and 
associated algebras,} Carleton Math. Lect. Notes, 
Vol. 5, 1973, pp. 1-119.    
\bibitem[DZ1]{DZ1} A. Dean and F. Zorzitto, 
{\it A criterion for pure simplicity,}
J.Algebra {\bf 132}, (1990), 50-71. 

\bibitem[DZ2]{DZ2}
A.Dean and F.Zorzitto,
{\it Infinite dimensional representations of $\tilde{D_{4}}$,}
Glasgow Math.J.32(1990)25-33.


\bibitem[E]{E} M. Enomoto,
{\it A construction of Hilbert representations of quivers,}
RIMS Kokyuroku,{\bf1893}(2014),102-114.
\bibitem[EW1]{EW1} M. Enomoto and Y. Watatani, 
{\it Relative position of four subspaces in a Hilbert space,} 
Adv. Math. {\bf201} (2006), 263-317.  

\bibitem[EW2]{EW2} M. Enomoto and Y. Watatani, 
{\it Exotic indecomposable systems of four subspaces 
in a Hilbert space,}
 Integral Equations Operator Theory {\bf59} (2007), 
 149-164. 

\bibitem[EW3]{EW3} M. Enomoto and Y. Watatani, 
{\it Indecomposable representations of quivers on 
infinite-dimensional Hilbert spaces,}
J. Funct. Anal. {\bf 256} (2009), 959-991. 
 
\bibitem[EW4]{EW4}
M.Enomoto and Y.Watatani, Strongly irreducible operators and indecomposable  
representations of quivers on 
infinite-dimensional Hilbert spaces,
  Integral Equations Operator Theory {\bf 83}(2015),563-587. 



\bibitem[F]{F} U. Fixman, 
{\it On algebraic equivalence between pairs of linear transformations,}
 Trans.Amer. Math. Soc. {\bf 113} (1964), 424-453.
\bibitem[FO]{FO} U. Fixman
and F.Okoh,
{\it Extensions of  pairs of linear transformations 
between infinite dimensional vector spaces,}
Linear algebra Appl.19 (1978) 275-292.
\bibitem[FZ]{FZ} U. Fixman
and F.A.Zorzitto,
{\it A purity criterion for pairs of linear transformations,}
Canad.J.Math. 26 (1974) 734-745.
\bibitem[FiW]{FiW} 
P.Fillmore and J.Williams,
On operator ranges,
Adv.Math.7(1971),254-281.

\bibitem[Ga]{Ga} P. Gabriel, {\it Unzerlegbare Darstellungen I,} 
Manuscripta Math. 6 (1972), 71-103.
\bibitem[GP]{GP} I. M. Gelfand and V. A. Ponomarev,
{\it Problems of linear algebra and
classification of quadruples of subspaces in a finite-dimensional vector
space,} Coll. Math. Spc. Bolyai 5, Tihany (1970), 163-237.

\bibitem[GR]{GR} P. Gabriel and A.V. Roiter, Representations of 
Finite-Dimensional Algebras, Springer-Verlag, 1997. 

\bibitem[Gi]{Gi} F. Gilfeather, 
{\it Strong reducibility of operators,} 
Indiana Univ. Math. J. 22, (1972), 393-397. 


\bibitem[GHJ]{GHJ} F. Goodman, P. de la Harpe and V. Jones , 
Coxeter Graphs and Towers of Algebras, MSRI Publications, 14, 
Springer, Berlin, 1989. 

\bibitem[H]{H}P.R.Halmos,Ten problems in Hilbert space,Bull.Amer.Math.Soc.76(1970),887-933.
\bibitem[HRR]{HRR} K.J. Harrison, H. Radjavi and P. Rosenthal,
{\it A transitive medial subspace lattice}, 
Proc. Amer. Math. Soc. 28 (1971), 119-121.

\bibitem[JW1]{JW1} C.Jiang and Z. Wang, 
{\it Strongly Irreducible Operators on Hilbert Space}, 
Longman, 1998.

\bibitem[JW2]{JW2} C.Jiang and Z. Wang, 
{\it Structure of Hilbert Space Operators}, 
World Scientific, 2006.
\bibitem[Jo]{Jo} V.Jones, 
Index for subfactors,Invent.Math.72(1983)1-25.
\bibitem[Ka]{Ka} V. G. Kac, {\it Infinite root systems,  
representations of graphs and invariant theory, } 
Invent. Math. 56 (1980), 57-92. 
\bibitem[Kau]{Kau}W.E.Kaufman,
Representing a closed operator as a quotient of continuous operators, Proc.Amer. Math. Soc. 72 (1978), 531-534. 
\bibitem[Ko]{Ko}
H.Kosaki,
On intersections of domains of unbounded positive operators,
Kyushu J.Math. 60(2006)3-25. 
\bibitem[KR]{KR} H. Krause and C. M. Ringel, ed., 
{\it Infinite Length Modules,} Birkh\"{a}user, 2000. 
\bibitem[Kro]{Kro}L.Kronecker,
{\it Algebraische Reduktion der Scharen bilinearer Formen,}
Sitzungsber. Akad.Berlin(1890),1225-1237,Jbuch.22,169.
\bibitem[KRo]{KRo} S. A. Kruglyak and A. V. Roiter, 
{\it Locally scalar representations of graphs in the 
category of Hilbert spaces,} Funct. Anal. Appl. 39 (2005), 
91-105.   


\bibitem[KRS]{KRS} S. Kruglyak, V. Rabanovich and Y. Samoilenko, 
{\it On sums of projections,} Functional Anal. Appl., 36(2002), 
182-195. 

\bibitem[MS]{MS} Y. P. Moskaleva and Y. S. Samoilenko, 
{\it Systems of $n$ subspaces and representations of *-algebras 
generated by projections,} Methods Funct. Anal. Topology 12 (2006),
 no. 1, 57-73.

\bibitem[Na1]{Na1} L. A. Nazarova,{\it Representations of quadruples, }
Izv. Akad. Nauk 
SSSR Ser. Mat. 31 (1967), 1361-1377. 

\bibitem[Na2]{Na2} L. A. Nazarova, {\it Representation of quivers 
of infinite type,} Izv. Akad. Nauk SSSR, Ser. Mat., 37 (1973),
 752-791. 

\bibitem[Ok]{Ok} F. Okoh, 
{\it Applications of linear functional to Kronecker modules I,}
 Linear Algebra Appl.76 (1986), 165-204.

\bibitem[RR]{RR}H. Radjavi and P. Rosenthal, {\it Invariant Subspaces}, 
Springer-Verlag, 1973. 

\bibitem[Ri1]{Ri1} C.M. Ringel, {\it Infinite dimensional representations 
of finite dimensional hereditary algebras}, 
Symposia Mathematica 23 (1979), 
Istituto Naz. Alta Matematica. 321-412.  

\bibitem[Ri2]{Ri2}
C.M.Ringel,The rational invariants of the tame quivers,
inventiones math.58,217-239(1980).



\bibitem[Ri3]{Ri3} C.M. Ringel, {\it Infinite length modules, 
Some examples as introduction, } 
Infinite Length Modules,
 Birkh\"{a}user, Basel, 2000, 1-73. 

\bibitem[Sh]{Sh} A. L. Shields, {\it Weighted shift operators and 
analytic function theory, } 
Topics in OperatorTheory, Math. Surveys Monographs, 
Amer. Math. Soc., Providence, RI 13 (1974), 49-128






 


\end{thebibliography}
\end{document}